\documentclass[reqno,b5paper]{amsart}
\usepackage{amssymb}
\usepackage{dsfont}
\usepackage{color}

\newtheorem{theorem}{Theorem}[section]
\newtheorem{problem}[theorem]{Problem}
\newtheorem{proposition}[theorem]{Proposition}
\newtheorem{corollary}[theorem]{Corollary}
\newtheorem{lemma}[theorem]{Lemma}
\newtheorem{example}[theorem]{Example}
\newtheorem{remark}[theorem]{Remark}
\newtheorem{definition}[theorem]{Definition}
\newtheorem{fact}[theorem]{Fact}

\newcommand{\T}{\mathbb{T}}
\newcommand{\Z}{\mathbb{Z}}
\newcommand{\N}{\mathbb{N}}

\newcommand{\supp}{\mathrm{supp}}

\catcode`\@=12
\def\T{{\mathbb T}}
\def\eps{{\varepsilon}}

\def\Z{{\mathbb Z}}
\def\N{{\mathbb N}}
\def\R{{\mathbb R}}
\def\Q{{\mathbb Q}}
\def\P{{\mathbb P}}

\makeindex

\begin{document}
\title[$s$-characterized subgroups]{Statistically characterized subgroups related to some non-arithmetic sequence of integers II (a quest for countable subgroups)}
\subjclass[2010]{Primary: 22B05, Secondary: 11B05, 40A05} \keywords{Circle group, characterized subgroup, natural density, s-characterized subgroup, arithmetic sequence, lifting function}

\author{Pratulananda Das}
\address{Department of Mathematics, Jadavpur University, Kolkata-700032, India}
\email {pratulananda@yahoo.co.in}

\author{Ayan Ghosh}
\address{Department of Mathematics, Jadavpur University, Kolkata-700032, India}
\email {ayanghosh.jumath@gmail.com}

\author{Tamim Aziz}
\address{Department of Mathematics, University of North Bengal, Darjeeling-734013, India}
\email {rs\_tamim@nbu.ac.in}

\begin{abstract}
Following the work of [Dikranjan et al., Fund. Math. 249:185-209, 2020] for arithmetic sequences, very recently in [Das et al., Expo. Math. 43(3):125653, 2025], statistically characterized subgroups have been investigated for certain types of  non-arithmetic sequences. Building on this work, we investigate further and demonstrate that, for a particular class of non-arithmetic sequences, the statistically characterized subgroup coincides with the corresponding characterized subgroup. In this context it should be kept in mind that statistical convergence (convergence w.r. to the ideal of natural density zero sets) encompasses much more sequences than usual convergence (convergence w.r. to the ideal of finite sets) and it had already been shown that statistically characterized subgroups corresponding to arithmetic sequences can not be characterized by any sequence [Das et al., Bull. Sci. Math. 179(2):103157, 2022] and they are always of the size of the continuum. From the very beginning it has been an open question as to whether statistically characterized subgroups can be small in size i.e. countably infinite. Our observation thus sheds new light on the crucial role of sequences generating subgroups of the circle group and at the same time  one can subsequently identify a class of sequences for which statistically characterized subgroups are countably infinite. This result provides a negative solution to Problem 2.16 posed in [Das et al., Expo. Math. 43(3):125653, 2025] and Question 6.3 from [Dikranjan et al., Fund. Math. 249:185-209, 2020]. Additionally, our findings resolve several open problems from [Dikranjan et al., Topo. Appl., 2025].
\end{abstract}
\maketitle

\section{Introduction and background}
The concept of characterized subgroups has evolved significantly over the years, broadening its scope beyond its original foundations to serve as a generalization of the torsion subgroup. Its rich history is closely tied to the study of sequences of multiples of a real number mod 1. The classical case, where these sequences are uniformly distributed mod 1, was first explored in the seminal work of Herman Weyl \cite{W}, and has since inspired further developments (see \cite{BDBW, BDMW1, DGDis, DPS} for more details and recent advancements). Moreover, these sequences of multiples have a deep connection with Ergodic theory, as can be seen from the study of Sturmian sequences and Hartman sets \cite{Wi}. The concept also plays a pivotal role in the structure theory of locally compact abelian groups \cite{AT,AD,R}. Furthermore, there is a profound relationship between characterized subgroups and ``trigonometric thin sets", particularly Arbault sets \cite{A1}, which are significant in Harmonic Analysis (see \cite{BuKR, DG, El, Ka} for detailed investigations in these directions).

Before proceeding further, let us formally present the definition of a {\em characterized subgroup} of the circle group $\T$.

\begin{definition}
Let $(a_n)$ be a sequence of integers, the subgroup
$$
t_{(a_n)}(\T) := \{x\in \T: a_nx \to 0\mbox{ in } \T\}
$$
of $\T$ is called {\em a characterized} $($by $(a_n))$ {\em subgroup} of $\T$.
\end{definition}

The term {\em characterized} appeared much later, coined in \cite{BDS}. Further, it is important to note that, for practical purposes, it is sufficient to work with sequences of positive integers only, which has been the usual practice. Although these subgroups can be generated by any sequence of positive integers, one major theme over the years has been the study of these subgroups when they are generated by arithmetic sequences. Precisely, a sequence of positive integers $(a_n)$ is an
{\em arithmetic sequence} if
$$1 = a_0 < a_1 < a_2 < \dots < a_n < \dots ~~\mbox{and}~a_n|a_{n+1}~ \mbox{for every}~n \in \mathbb{N}.$$

In 2020, the notion of characterized subgroups were generalized by using a more general mode of convergence and here the idea of natural density (see \cite{Bu1}) and the corresponding mode of convergence came into the picture. In the language of \cite{DDB} "Although the correspondence $(a_n) \to t_{(a_n)}(\T)$ is decreasing (with respect to inclusion), in many cases the subgroup $t_{(a_n)}(\T)$ is rather small, even if the sequence $(a_n)$ is not too dense." So it seemed a natural course of action to involve a more general mode of convergence (for more details of the reasons and motivation behind this approach see \cite{DDB}).

For $m,n\in\mathbb{N}$ and $m\leq n$, let $[m, n]$ denote the set $\{m, m+1, m+2,...,n\}$. By $|A|$ we denote the cardinality of a set $A$. The lower and upper natural densities of $A \subseteq \mathbb{N}$ are defined by
$$
\underline{d}(A)=\displaystyle{\liminf_{n\to\infty}}\frac{|A\cap [0,n-1]|}{n} ~~\mbox{and}~~
\overline{d}(A)=\displaystyle{\limsup_{n\to\infty}}\frac{|A\cap [0,n-1]|}{n}.
$$
If $\underline{d}(A)=\overline{d}(A)$, we say that the natural density of $A$ exists and it is denoted by $d(A)$. As usual,
$$
\mathcal{I}_d = \{A \subseteq \mathbb{N}: d(A) = 0\}
$$
denotes the ideal of ``natural density zero" sets and $\mathcal{I}_d^*$ is the dual filter, i.e., $\mathcal{I}_d^* = \{A \subseteq \mathbb{N}: d(A) = 1\}$. Let us now recall the notion of statistical convergence in the sense of \cite{F,Fr,S,St,Z}.

\begin{definition}\label{Def1}
A sequence of real numbers $(x_n)$ is said to converge to a real number $x_0$ statistically if for any $\eps > 0$, $d(\{n \in \mathbb{N}: |x_n - x_0| \geq \eps\}) = 0$.
\end{definition}

It is now known \cite{S} that $x_n \to x_0$ statistically precisely when there exists a subset A of $ \N$ of asymptotic density 0, such that $\displaystyle{\lim_{n \in \N \setminus A}} x_n = x_0$ which makes this convergence interesting but not too wild. The concept of statistical convergence has been extensively studied over the years, beginning in metric spaces and later expanding to general topological spaces \cite{MK}. Over the past three decades, significant progress has been made in this area, largely due to its natural extension of the notion of usual convergence. Statistical convergence retains many fundamental properties of traditional convergence, while encompassing a broader class of sequences. For some insightful applications of statistical convergence, the readers are referred to \cite{BDK,CKG,FGT}. As a natural application, the following notion was introduced in \cite{DDB}:
\begin{definition}
For a sequence of integers $(a_n)$ the subgroup
\begin{equation}\label{def:stat:conv}
t^s_{(a_n)}(\T) := \{x\in \T: a_nx \to 0\  \mbox{ statistically in }\  \T\}
\end{equation}
of $\T$ is called {\em a statistically characterized} (shortly, {\em an s-characterized}) $($by $(a_n))$ {\em subgroup} of $\T$.
\end{definition}

The subsequent result justifies the investigation of this new notion of s-characterized subgroups, as it turns out that, though in general larger in size, these subgroups are still essentially topologically nice.

\begin{theorem}\label{theoremA}\cite[Theorem A]{DDB}
For any increasing sequence of integers $(a_n)$, the subgroup $t^s_{(a_n)}(\T)$ is a $F_{\sigma\delta}$ (hence, Borel) subgroup of $\T$ containing $t_{(a_n)}(\T)$.
\end{theorem}

The following theorem leads to a general assertion about the size of s-characterized subgroups for arithmetic sequences.

\begin{theorem}\label{theoremB}\cite[Theorem B]{DDB}
Let $(a_n)$ be an arithmetic sequence. Then $|t^s_{(a_n)}(\T)|=\mathfrak c$.
\end{theorem}

As a consequence, we find the more general result that the new subgroup $t^s_{(a_n)}(\T)$ always
differs from the subgroup $t_{(a_n)}(\T)$ for arithmetic sequences.

\begin{theorem}\label{theoremC}\cite[Theorem C]{DDB}
For any arithmetic sequence $(a_n)$, $t^s_{(a_n)}(\T)\neq t_{(a_n)}(\T)$.
\end{theorem}

Unlike Theorem A, both Theorem B and Theorem C rely on  arithmetic sequence. So, naturally the following problem was asked in \cite{DDB}.

\begin{problem}\label{prob1}\cite[Question 6.3]{DDB}
Do Theorems B and C hold true for arbitrary sequences $(u_n)$?
\end{problem}
Our next observation was carried out in \cite{DG} which had provided several examples of uncountable $F_{\sigma\delta}$ subgroups that cannot be characterized.
\begin{theorem}\label{ththin}\cite[Theorem 2.7]{DG}
For any arithmetic sequence $(a_n)$, the subgroup $t^s_{(a_n)}(\T)$ cannot be characterized.
\end{theorem}
Motivated by these observations, a series of open problems were recently posed in \cite{DDBH}.
\begin{problem}\label{prob2}\cite[Question 7.16]{DDBH}
Does there exist an infinite statistically characterized subgroup that can be characterized?
\end{problem}
\begin{problem}\label{prob3}\cite[Question 7.17]{DDBH}
Does there exist an increasing sequence $(u_n)$ such that $t^s_{(u_n)}(\T)$ is countably infinite?
\end{problem}

Coming back to the literature regarding characterized subgroups, in an interesting departure, a non-arithmetic sequence $(\zeta_n)$ was defined in \cite{DK} as follows:
\begin{equation}\label{eqnonarith}
1,2,4,6,12, 18, 24,  \ldots, n!, 2\cdot n!, 3 \cdot n!, \ldots , n \cdot n!, (n+1)!, \ldots
\end{equation}
It was established in \cite{DK} that $t_{(\zeta_n)}(\T) = \Q/\Z$. Motivated by this observation, for an arithmetic sequence $(a_n)$ the following general class of non-arithmetic sequences was introduced in \cite{DG8}. Let $(d_n^{a_n})$ be an increasing sequence of integers formed by the elements of the set,
\begin{equation}\label{nonarithdef}
\{ra_k \ : \ 1\leq r< b_{k+1}\}.
\end{equation}
When there is no confusion regarding the sequence $(a_n)$, we simply denote this sequence by $(d_n)$. Note that for $a_n=n!$ corresponding non-arithmetic sequence $(d_n)$ coincides with the sequence $(\zeta_n)$. A through investigation regarding the subgroup $t_{(d_n)}(\T)$, in particular cardinality and structure related questions and its relation with $t_{(a_n)}(\T)$ was carried out in \cite{DG8}.

As a natural consequence of the works done in \cite{DG8} and \cite{DDB}, the subgroup $t^s_{(d_n)}(\T)$ was considered in \cite{DG9}. One of the main results of \cite{DG9} was the following interesting observation, which provides a negative solution to Question 6.6 posed in \cite{DDB}.
\begin{theorem}\label{sconth}\cite[Theorem 2.14]{DG9}
Let $(a_n)$ be an arithmetic sequence such that for each $m\in\N$, $\lim\limits_{n\to\infty}\frac{\sum\limits_{i=0}^{m-1} (b_{n-i}-1)}{\sum\limits_{i=1}^n (b_i-1)}=0$. Then $|t^s_{(d_n)}(\T)|=\mathfrak{c}$.
\end{theorem}
Motivated by this observation, the following open problem was posed in \cite{DG9}.
\begin{problem}\label{prob4}\cite[Problem 2.16]{DG9}
Is $|t^s_{(d_n)}(\T)|=\mathfrak{c}$, for an arbitrary arithmetic sequence of integers $(a_n)$?
\end{problem}
In this article, our primary focus is on the ``statistically characterized subgroups" generated by the sequences $(d_n)$, i.e., $t^s_{(d_n)}(\T)$ and to understand the cardinality pattern of the statistically characterized subgroup generated by the sequence $(d_n)$. The main tool in our quest is the interesting new idea of a lifting function $L$ corresponding to a subsequence $(a_{n_k})$ of a sequence of positive integers $(a_n)$. More precisely the lifting function $L \ : \ \mathcal{P}(\N) \to \mathcal{P}(\N)$ is defined as
$$
L(A)=\bigcup\limits_{k\in A} [n_{k-1},n_{k}-1]
$$
for any $A \subset \mathbb{N}$. In Theorem \ref{sconthmain1}, we have been able to show that for arithmetic sequences satisfying some specific properties, they are always of size $\mathfrak{c}$. This in turn provides a more general view of \cite[Theorem 2.14]{DG9}. For example it can be seen that  $|t^{s}_{(d_n)}{(\mathbb{T})}|=\mathfrak{c}$ for the sequence $(n!)$. Next we dwell upon Problem \ref{prob2} (stated above) looking for instances where one can find a positive solution. This brings us to the main observation of this article, namely, Theorem \ref{sconthmain2} which states that for arithmetic sequences satisfying certain properties, the subgroup $t^s_{(d_n)}(\T)$ actually coincides with the subgroup $t_{(d_n)}(\T)$. As a corollary of this result we obtain another interesting observation that these statistically characterized subgroups are, in fact, countable (see Corollary \ref{sconcoromain2}). For a concrete example one can take the arithmetic sequence $(2^{\frac{n(n+1)}{2}})$. Therefore, Theorem \ref{sconthmain2} provides a positive solution to Problem \ref{prob2} whereas Corollary \ref{sconcoromain2} provides a negative solution to Problem \ref{prob1} and Problem \ref{prob4} while providing a positive solution to Problem \ref{prob3}. Finally, we conclude this article by presenting some characterizable conditions and open problems.

In this article by $(d_n)$ we would always mean the sequence defined in Eq (\ref{nonarithdef}) corresponding to the arithmetic sequence $(a_n)$ unless otherwise stated.

\section{Notation and terminology.\vspace{.3cm} \\ }
Throughout $\R$, $\Q$, $\Z$, $\P$ and $\N$ will stand for the set of all real numbers, the set of all rational numbers,
the set of all integers, the set of primes and the set of all natural numbers (note that we do not consider zero as a natural number) respectively. The first three are equipped with their usual abelian group structure and the circle group $\T$ is identified with the quotient group $\R/\Z$ of $\R$ endowed with its usual compact topology.

Following \cite{Ka}, we may identify $\T$ with the interval [0,1] identifying 0 and 1. Any real valued function $f$ defined on $\T$ can be identified with a periodic function defined on the whole real line $\R$ with period 1, i.e., $f(x+1)=f(x)$ for every real $x$. When referring to a set $X\subseteq \T$ we assume that $X\subseteq [0,1]$ and $0\in X$ if and only if $1\in X$. For a real $x$, we denote its fractional part by $\{x\}$ and the distance from the integers by $\|x\|=\min\big\{\{x\},1-\{x\}\big\}$.

For arithmetic sequences, the following facts (see \cite{DG9}) will be used in this sequel time and again. So, before moving onto our main results here we recapitulate that once.
\begin{fact}\label{lemmanew}\cite{DI1}
For any arithmetic sequence $(a_n)$ and $x\in\T$, we can find a unique sequence $c_n\in [0,b_n-1]$ such that
\begin{equation}\label{canonical:repr}
x=\sum\limits_{n=1}^{\infty}\frac{c_n}{a_n},
\end{equation}
where $c_n<b_n-1$ for infinitely many $n$.
\end{fact}
%\begin{proof}
%For better clarity we recall the construction of the sequence $(c_n)$. Consider $c_1=\lfloor a_1x\rfloor$, where $\lfloor \ \rfloor$ denotes the integer part. Therefore, $x-\frac{c_1}{a_1}<\frac{1}{a_1}$.
%
%Suppose, $c_1,c_2,\ldots,c_k$ are defined for some $k\geq 1$ with $x_k=\sum\limits_{n=1}^{k}\frac{c_n}{a_n}$ and $x-x_k<\frac{1}{a_k}$. Then the $(k+1)$-th element is defined as $c_{k+1}=\lfloor a_{k+1}(x-x_k)\rfloor$.
%\end{proof}

For $x\in\T$ with canonical representation (\ref{canonical:repr}), we define
\begin{itemize}
\item[$\bullet$] $supp(x) = \{n\in \N: c_n \neq 0\}$,
\item[$\bullet$] $supp_q(x)=\{n\in\N\ : \ c_n=b_n-1\}$.
\end{itemize}
For an arithmetic sequence of integers $(a_n)$, the sequence of ratios $(b_n)$ is defined as
$$
b_1=a_1 \mbox{     \ and \      } b_n=\frac{a_n}{a_{n-1}}  \mbox{ for  } n \geq 2.
$$
Note that, $A\subseteq\N$ is called
\begin{itemize}
\item[(i)] $b$-bounded if the sequence of ratios $(b_n)_{n\in A}$ is bounded.
\item[(ii)] $b$-divergent if the sequence of ratios $(b_n)_{n\in A}$  diverges to $\infty$.
\end{itemize}
We say $(a_n)$ is $b$-bounded ($b$-divergent) if $\N$ is $b$-bounded ($b$-divergent). Also, note that for each $j\in\N$,
\begin{equation}\label{eqsum}
\sum\limits_{i=j}^\infty \frac{c_i}{a_i} \leq \sum\limits_{i=j}^\infty \frac{b_i-1}{a_i} = \sum\limits_{i=j}^\infty \bigg(\frac{1}{a_{i-1}} - \frac{1}{a_i} \bigg) \leq \frac{1}{a_{j-1}}.
\end{equation}
%%%%%%%%%%%%%%%%%%%%%%%%%%%%%%%%%%%%%%%%%%%%%%%%%%%%%%%%%%
\begin{lemma}\label{uconlmain}\cite[Lemma 3.1]{DR}
For $x\in\T$ with canonical representation (\ref{canonical:repr}), for every natural $n > 1$ and every non-negative integer $t$,
\begin{equation}\label{eqlemain1}
\{a_{n-1}x\}=\frac{c_n}{b_n}+\frac{c_{n+1}}{b_nb_{n+1}}+\ldots+\frac{c_{n+t}}{b_n\ldots b_{n+t}}+\frac{\{a_{n+t}x\}}{b_n\ldots b_{n+t}}.
\end{equation}
In particular, for $t=1$, we get
\begin{equation}\label{eqlemain2}
\{a_{n-1}x\}=\frac{c_n}{b_n}+\frac{c_{n+1}}{b_nb_{n+1}}+\frac{\{a_{n+1}x\}}{b_nb_{n+1}}.
\end{equation}
\end{lemma}
%%%%%%%%%%%%%%%%%%%%%%%%%%%%%%%%%%%%%%%%%%%%%%%%%%%%%%%%%%
\begin{fact}\label{eqInorm}
The following facts are well known for integer norm:
\begin{itemize}
\item for any $A\subseteq\N$ and $y_n=x_n+z_n$ if $\lim\limits_{n\in A} z_n=0$ then $\lim\limits_{n\in A}\|y_n\|=\lim\limits_{n\in A}\|x_n\|$.
\item for any integer $n$ and for any real number $x$, $\|n+x\|=\|x\|$.
\end{itemize}
\end{fact}
Also note that for any $r\in\N$,
\begin{equation}\label{eqr}
\bullet \ \mbox{if } \ r\{a_nx\}<1\ \mbox{ then } \ \{ra_nx\}=r\{a_nx\}.
\end{equation}
\begin{equation}\label{eqrn}
\bullet \ \mbox{if } \ r\|a_nx\|<\frac{1}{2} \ \mbox{ then } \ \|ra_nx\|=r\|a_nx\|.
\end{equation}
\begin{remark}\label{uconr1}
Observe that $(a_n)$ is a subsequence of $(d_n)$. Therefore, we can write $a_k=d_{n_k}$. Now, from the construction of $(d_n)$, it is evident that
$n_{k+1}-n_{k}=b_{k+1}-1$ and $d_{n_k+r-1}=ra_k$, where $r\in[1,b_{k+1}-1]$.
\end{remark}
\begin{definition}\label{uconr2}
For $A\subseteq\mathbb{N}$, let us define the lifting function $L \ : \ \mathcal{P}(\N) \to \mathcal{P}(\N)$ as
   $$
L(A)=\bigcup\limits_{k\in A} [n_{k-1},n_{k}-1].
$$
Then, observe that $L$ is injective and for a sequence $(A_n)$ of sets in $\mathbb{N}$, we can write
\begin{itemize}
\item $L(\displaystyle\bigcup_{n=1}^{\infty}A_n)=\displaystyle\bigcup_{n=1}^{\infty}L(A_n)$ and $L(\displaystyle\bigcap_{n=1}^{\infty}A_n)=\displaystyle\bigcap_{n=1}^{\infty}L(A_n)$,
\item $L(A_m\setminus A_n)=L(A_m)\setminus L(A_n)$ for all $m,n\in\mathbb{N}$.
\end{itemize}
\end{definition}
\vspace{.1cm}

\section{Main results.\vspace{.3cm} \\ }

As mentioned before, our primary focus is to understand the nature of the cardinality of the subgroup $t^{s}_{(d_n)}{(\T)}$. As $(d_n)$ is a non-arithmetic sequence the hunch has been that it may not influence an uniform behavior for correspondingly generated $s$-characterized subgroups (unlike arithmetic sequences). Moreover as it is generated from an arithmetic sequence but arithmetic sequences themselves can be quite dense or quite non-dense, per say, it seems natural to expect what we can call "extreme variations". As we are looking into $s$-characterized subgroups where natural density is playing a major role, it is imperative to somehow involve this notion which will become clearer in the next two sections.

\vspace{.3cm}
\subsection{Sequences $(d_n)$ generating uncountable s-characterized subgroups. \vspace{.3cm} \\}
Before going further we provide some interesting definitions related to arithmetic sequences which will be crucial to formulate our main results.

\begin{definition}
An arithmetic sequence $(a_n)$ is called density lifting invariant (in short, dli) if for every infinite $A\subseteq\N$ with $d(A)=0$, ${d}(L(A))=0$.
\end{definition}
For $A\subseteq\mathbb{N}$ and $m\in\mathbb{N}\cup \{0\}$, we write
$$A-m=\{n\in\mathbb{N}:n=a-m~\mbox{ for some }~a\in A\}.$$

\begin{definition}
An arithmetic sequence $(a_n)$ is called weakly density lifting invariant (in short, weakly dli) if there exists an infinite $A\subseteq\N$ such that ${d}(L(A-m))=0$ for each $m\in\N\cup \{0\}$.
\end{definition}

\begin{proposition}\label{prodliwdli}
Any dli arithmetic sequence is weakly dli.
\end{proposition}
\begin{proof}
Let $(a_n)$ be a dli arithmetic sequence. Fix any infinite $A\subseteq\N $ such that $d(A)=0$. Since $d$ is translation invariant, $d(A-m)=0$ for each $m\in \mathbb{N}\cup \{0\}$. As $(a_n)$ is dli, it follows that ${d}(L(A-m))=0$ for each $m\in\mathbb{N}\cup \{0\}$.  Consequently, $(a_n)$ is weakly dli.
\end{proof}
Our next lemma is instrumental in proof of Proposition \ref{prodliqbound} which provides a broader view of dli arithmetic sequences.
\begin{lemma}\label{ledliqboundnew}
For any $b$-bounded $A\subseteq\N$, if $d(A)=0$ then $d(L(A))=0$.
\end{lemma}
\begin{proof}
Let $(a_n)$ be an arithmetic sequence of integers and $A$ be a $b$-bounded subset of $\N$ such that $d(A)=0$. Then there exists $M\in\mathbb{N}$ such that $2\leq b_n\leq M$ for each $n\in A$. Now Remark \ref{uconr1} entails that for each $k\in A$
$$
a_k=d_{n_k}~\mbox{ and }~ n_{k}-n_{k-1}=b_{k}-1<M.
$$
Let us set $A'=\{n_{k-1}:k\in A\}$. It is clear that $n_k\geq k$ for each $k\in\mathbb{N}$. Therefore, observe that
\begin{align*}
	\bar{d}(A')&=\limsup_{n\to\infty}\frac{|\{j\in\mathbb{N}:j\leq n~\mbox{and}~j\in A'\}|}{n}\\
	 &\leq \limsup_{n\to\infty}\frac{|\{j\in\mathbb{N}:j\leq n~\mbox{and}~j\in A\}|}{n}\\
	&=\bar{d}(A)=0.
\end{align*}
It is easy to realize that $L(A)\subseteq\displaystyle\bigcup_{i=0}^{M}(A'+i)$. Since $d(A')=0$, we obtain $d(A'+i)=0$ for each $i\in\mathbb{N}$. This also ensures that $d(L(A))=0$.
\end{proof}
\begin{proposition}\label{prodliqbound}
Any $b$-bounded arithmetic sequence is dli but not conversely.
\end{proposition}
 \begin{proof}
Assume that $(a_n)$ is a $b$-bounded arithmetic sequence. Pick any $A\subseteq\mathbb{N}$ such that $d(A)=0$. As $A$ is again $b$-bounded, Lemma \ref{ledliqboundnew} ensures that $d(L(A))=0$. Consequently,  $(a_n)$ is dli.\\

To illustrate that the converse is not necessarily true, we will construct an arithmetic sequence $(a_n)$ which will be dli but not $b$-bounded. Consider a \( K \subseteq \mathbb{N} \) such that
$$
K=\bigcup_{j=1}^{\infty}[g_j,h_j]
$$
where $g_1=1, g_j\leq h_j,|g_{j+1}-h_j|\to\infty$, and $d(K)=1$ (note that the existence of such a set can always be ensured, for example, by taking $h_{j}-g_j=j^3$ and $g_{j+1}-h_j=j$). Let us rewrite
$$
K=\{1=n_0<n_1<n_2<...<n_k<...\}.
$$
From the construction of $K$, one can easily obtain a sequence $(s_j)$ such that $n_{s_j}=h_j$ and $n_{{s_j}+1}=g_{j+1}$.
For each $k\in\mathbb{N}\cup\{0\}$, we set
$$
b_{k+1}=n_{k+1}-n_{k}+1.
$$
Then, observe that
\begin{equation*}
b_k =
\left\{
\begin{array}{lr}
g_{j+1}-h_{j}+1&\mbox{ if }~k=s_{j}+1,\\
2&\mbox{otherwise}.
\end{array}
\right.
\end{equation*}
Thus, the associated arithmetic sequence $(a_n)$ is given by
\begin{equation*}
    a_0=1~\mbox{ and }~a_{n+1}=b_{n+1}a_n.
\end{equation*}
Since $(b_{s_j+1})$ is divergent, it is evident that $(a_n)$ is not $b$-bounded.

Now, let us pick any infinite $A\subseteq\mathbb{N}$ with $d(A)=0$. If $A$ is $b$-bounded then from Lemma \ref{ledliqboundnew} it follows that $d(L(A))=0$. So, let us assume that $A$ is not $b$-bounded. From the construction of $(b_n)$, one can observe that there actually exist non-empty sets $B,I\subseteq\mathbb{N}$ with $B$ being $b$-bounded and $I$ being $b$-divergent such that $A=B\cup I$. Since $B\subseteq A$ and $d(A)=0$ we already have $d(B)=0$. Therefore, from Lemma \ref{ledliqboundnew}, it is evident that $d(L(B))=0$. Note that  $I\subseteq\{s_j+1:j\in\mathbb{N}\}$. Also, we have
$$
L(\{s_j+1:j\in\mathbb{N}\}) = \displaystyle\bigcup_{j=1}^{\infty}[h_j,g_{j+1}-1] \subseteq (\N\setminus K)\cup \{h_j:j\in\mathbb{N}\},
$$
and consequently,
\begin{eqnarray*}
d(L(I)) \leq d(L(\{s_j+1:j\in\mathbb{N}\})) &\leq& d((\mathbb{N}\setminus K) \cup \{h_j:j\in\mathbb{N}\}) \\ &\leq& d(\mathbb{N}\setminus K) +d(\{h_j:j\in\mathbb{N}\}) \\ &=& d(\{h_j:j\in\mathbb{N}\})\ \mbox{ (since $d(K)=1)$} \\ &=& 0 \ \mbox{ (since $|h_{k+1}-h_{k}|\to\infty$)}.
\end{eqnarray*}
Finally, in view of Remark \ref{uconr2}, we can conclude that $d(L(A))=0$. Thus, $(a_n)$ is the required dli arithmetic sequence which is not $b$-bounded.
\end{proof}

\begin{proposition}\label{prowdliqbound}
Let $(a_n)$ be an arithmetic sequence such that $\lim\limits_{n\to\infty}\frac{ b_{n}}{\sum\limits_{i=1}^n (b_i-1)}=0$. Then $(a_n)$ is weakly dli.
\end{proposition}

\begin{proof}
Assume that $(a_n)$ is an arithmetic sequence for which \begin{equation}\label{wdl condition}
 \lim\limits_{n\to\infty}\frac{ b_{n}}{\sum\limits_{i=1}^n (b_i-1)}=0.
 \end{equation}
 Since $(a_n)$ is a subsequence of $(d_n)$, there exists $\{1=n_0<n_1<...<n_k<...\}$ such that $a_k=d_{n_k}$. In view of Remark \ref{uconr1}, we have
 $$
 n_{k}-n_{k-1}=b_{k}-1~\mbox{ for each }~k\in\mathbb{N}.
 $$
 Let us define the sequence $(u_n)$ recursively as follows:
	$$u_1=1~\mbox{ and }~u_{j+1}=\min\{r\in\mathbb{N}:~ r>u_{j}+j+1~\mbox{ with }~ n_{r}>j\sum_{i=1}^{j}\sum_{t=0}^{i} (b_{u_i+1-t}-1)\}.$$
 From the construction of $(u_n)$, it follows that
 \begin{equation} \label{Eq 12}
 n_{u_{j-1}}>(j-2)\displaystyle\sum_{i=1}^{j-2}\sum_{t=0}^{i} (b_{u_i+1-t}-1)~\mbox{ for each }~j\in\mathbb{N}\setminus\{1\}.
  \end{equation}
Set $A=\{u_j+1:j\in\mathbb{N}\}$. We claim that $d(L(A-m))=0$ for each $m\in\mathbb{N}\cup\{0\}$. Choose $m\in\mathbb{N}\cup\{0\}$. Then it is evident that $u_j>u_{j-1}+m+1~$ whenever $~j> m$. Let us set
 $$
 B_m=\displaystyle\bigcup_{j=m+1}^{\infty}[n_{u_j-m},n_{{u_j-m}+1}-1].
 $$
% $$A_m=[1,n_{u_{m}-m}-1]\cup\displaystyle\bigcup_{j=m}^{\infty}[n_{u_j-m},n_{{u_j-m}+1}-1].$$
First, observe that $L(A-m)\subset^* B_m$. Next note that
 %\begingroup
 \allowdisplaybreaks
 \begin{eqnarray*}
 \bar{d}(B_m)&=&\limsup_{n\to\infty}\frac{|B_m\cap[1,n]|}{n}\\
 &=&\limsup_{j\to\infty}\frac{|B_m\cap[1,n_{{u_j-m}+1}-1]|}{n_{{u_j-m}+1}-1}\\
%& \leq \limsup_{j\to\infty}\frac{|L(A-m)\cap[1,n_{u_{j-1}-m+1}-1]|}{n_{{u_j}-m+1}-1}+\limsup_{j\to\infty}\frac{|L(A-m)\cap[n_{u_j-m}, n_{{u_j}-m+1}-1]|}{n_{{u_j}-m+1}-1}\\
  &\leq&\limsup_{j\to\infty}\frac{|B_m\cap[1,n_{u_{j-2}-m+1}-1]|}{n_{{u_j}-m+1}-1} \\ & &+ \limsup_{j\to\infty}\frac{|B_m\cap[n_{u_{j-1}-m}, n_{{u_{j-1}}-m+1}-1]|}{n_{{u_j}-m+1}-1}\\
  & &+ \limsup_{j\to\infty}\frac{|B_m\cap[n_{u_j-m}, n_{{u_j}-m+1}-1]|}{n_{{u_j}-m+1}-1}\\
 &\leq&\limsup_{j\to\infty}\frac{\displaystyle\sum_{i=m}^{j-2} (b_{u_i-m+1}-1)}{n_{{u_{j-1}}}}+\limsup_{j\to\infty}\frac{(b_{u_{j-1}-m+1}-1)}{\{n_1-1+\displaystyle\sum_{i=2}^{u_j-m+1}(b_i-1)\}}
 \\
  & &+\limsup_{j\to\infty}\frac{(b_{u_{j}-m+1}-1)}{\{n_1-1+\displaystyle\sum_{i=2}^{u_j-m+1}(b_i-1)\}}\\
  &\leq& \limsup_{j\to\infty} \frac{\displaystyle\sum_{i=1}^{j-2}\sum_{t=0}^{i} (b_{u_i+1-t}-1)}{n_{u_{j-1}}}+\limsup_{j\to\infty}\frac{b_{u_{j-1}-m+1}}{\displaystyle\sum_{i=1}^{u_{j-1}-m+1}(b_i-1)}\\ & & +\limsup_{j\to\infty}\frac{b_{u_{j}-m+1}}{\displaystyle\sum_{i=1}^{u_j-m+1}(b_i-1)}\\
% &\leq \limsup_{j\to\infty} \frac{\displaystyle\sum_{i=1}^{j-1}\sum_{t=0}^{i} (b_{u_i+1-t}-1)}{n_{u_{j-1}}}+ \limsup_{j\to\infty}\frac{\displaystyle\sum_{i=0}^{m-1}(q_{u_j-m+1-i}-1)}{\displaystyle\sum_{i=2}^{u_j-m+1}(b_i-1)}\\
&\leq& \lim_{j\to\infty}\frac{1}{j-2}+0~\mbox{ (in view of Eq (\ref{wdl condition}) and Eq (\ref{Eq 12}))}\ =\ 0.
 \end{eqnarray*}
 Consequently, $d(B_m)=0$ and this ensures that $d(L(A-m))=0$ since $L(A-m)\subset^* B_m$. As \( m \in \mathbb{N}\cup\{0\} \) was chosen arbitrarily, we deduce that $(a_n)$ is weakly dli.
\end{proof}
% footnote text
\let\thefootnote\relax\footnotetext{For any two subsets $A$ and $B$ of $\mathbb{N}$ we will denote $A\subset^*B$ if $A\setminus B$ is finite and $A=^*B$ if $A\Delta B$ is finite.}

\begin{corollary}
Let $(a_n)$ be an arithmetic sequence such that for each $m\in\N$, $\lim\limits_{n\to\infty}\frac{\sum\limits_{i=0}^{m-1} (b_{n-i}-1)}{\sum\limits_{i=1}^n (b_i-1)}=0$. Then $(a_n)$ is weakly dli.
\end{corollary}
\begin{proof}
Note that the given condition clearly implies
$$
\lim\limits_{n\to\infty}\frac{ b_{n}}{\sum\limits_{i=1}^n (b_i-1)}=0,
$$
and hence the proof is a direct consequence of Proposition \ref{prodliqbound}.
\end{proof}

Our next result provides a more general view of \cite[Theorem 2.14]{DG9}.
\begin{theorem}\label{sconthmain1}
If $(a_n)$ is weakly dli then $|t^{s}_{(d_n)}{(\mathbb{T})}|=\mathfrak{c}$.
\end{theorem}
%%%%%%%%%%%%%%%%%%%%%%%%%%%%%%%%%%%%%%%%%%%%%%%%%%%%%%%%%%%%

\begin{proof}
Assume that  $(a_n)$ is weakly dli. Then there exists an infinite $A\subseteq \mathbb{N}$ such that $d(L(A-m))=0$ for each $m\in\mathbb{N}$. Let us write $A=\{u_1+1<u_2+1<...<u_k+1<...\}$.
Consider $x\in\mathbb{T}$ such that
$$
supp(x)\subseteq A.
$$
Let $0<\epsilon<\frac{1}{2}$ be given. Choose $m\in\mathbb{N}$ such that $\frac{1}{2^{m-2}}<\epsilon$. Set
$$
B=[1,n_{u_m-m+1}-1]\cup\bigcup_{j=m}^{\infty} [n_{u_j-m+1},n_{u_j+1}-1].
$$
Since $d(\displaystyle\bigcup_{i=0}^{m-1} L(A-i))=0$ and $B=^* \displaystyle \bigcup_{i=0}^{m-1} L(A-i)$, we can conclude that $d(B)=0$.
Note that, for sufficiently large $i\in \mathbb{N}\setminus B$, we have $i=n_{k}+r-1$ with $r\in[1,b_{k+1}-1]$ where $k\notin \displaystyle\bigcup_{j=m}^{\infty}[u_j-m+1,u_{j}+1]$, i.e., $k,k+1,...,k+m-1\notin supp(x)$.
Subsequently, we obtain
\begin{eqnarray*}
    r\{a_kx\}&\leq &ra_k\sum_{i=k+m}^{\infty}\frac{c_i}{a_i}\\
              &\leq & \frac{ra_k}{a_{k+m-1}} \leq \frac{r}{b_{k+1}}\frac{a_{k+1}}{a_{k+m-1}}\leq \frac{1}{2^{m-2}}<\epsilon.
\end{eqnarray*}
Note that in view of Eq (\ref{eqr}), we can write $\{ra_kx\}=r\{a_kx\}$. Then, for sufficiently large $i\in \mathbb{N}\setminus B$, we have
\begin{align*}
    \{d_ix\}=\{d_{n_k+r-1}x\}=\{ra_kx\}=r\{a_kx\}<\epsilon.
\end{align*}
This entails that $x\in t^{s}_{(d_n)}{(\mathbb{T})}$.\\

We now assert that indeed $t_{(d_n)}^{s}(\mathbb{T})$ contains $\mathfrak{c}$ many elements. To demonstrate this,  we fix a sequence $\zeta=(s_i)\in\{0,1\}^{\mathbb{N}}$. Now consider $B^{\zeta}=\displaystyle\bigcup_{k=1}^{\infty}\{u_{2k+s_k}+1\}$, i.e., the subset $B^{\zeta}$ of $(u_j+1)$ is obtained by considering at each stage $k$ either $u_{2k}+1$ or $u_{2k+1}+1$ depending on the choice imposed by $\zeta$. Note that for any two distinct  $\zeta, \eta\in \{0,1\}^{\mathbb{N}}$ we always have $B^{\zeta}\neq B^{\eta}$ and this produces an injection map defined as follows:
$$\{0,1\}^{\mathbb{N}}\ni\zeta\rightarrow B^{\zeta}.$$
For each $\zeta\in  \{0,1\}^{\mathbb{N}}$, we pick $x^{\zeta}\in\mathbb{T}$ such that $supp(x)=B^{\zeta}$. Evidently, $x^{\zeta}\in t_{(d_n)}^{s}(\mathbb{T})$.
Finally as $|\{0,1\}^{\mathbb{N}}|=\mathfrak{c}$, we can conclude that $|t_{(d_n)}^{s}(\mathbb{T})|=\mathfrak{c}$.
\end{proof}

\begin{corollary}\label{sconcoromain1}
If $(a_n)$ is weakly dli then $t^{s}_{(d_n)}{(\mathbb{T})}\neq t_{(d_n)}(\mathbb{T})$.
\end{corollary}
\begin{proof}
  Since $t_{(d_n)}(\mathbb{T})$ is countable \cite[Corollary 2.4, (iii)]{DG8}, the proof follows from Theorem \ref{sconthmain1}.
\end{proof}
\begin{example}
Consider the arithmetic sequence $(n!)$. Observe that the sequence $(n!)$ satisfies the condition specified in Proposition \ref{prodliqbound} and hene $(n!)$ is weakly dli. Consequently, from Theorem \ref{sconthmain1} it follows that $|t^{s}_{(d_n)}{(\mathbb{T})}|=\mathfrak{c}$.
\end{example}

\vspace{.3cm}
\subsection{Sequences $(d_n)$ generating countable s-characterized subgroups. \vspace{.3cm} \\}

\begin{definition}
An arithmetic sequence $(a_n)$ is said to be strongly non density lifting invariant (in short, strongly non dli) if for any infinite $A\subseteq\N$, $\bar{d}(L(A))>0$.
\end{definition}
Note that if $(a_n)$ is strongly non dli then it is not weakly dli. Then Proposition \ref{prodliwdli} ensures that $(a_n)$ is also not dli. In this regard, our next proposition provides a broader view of strongly non dli arithmetic sequences.
\begin{proposition}\label{prosndli}
Consider the arithmetic sequence $(a_n)$ such that $b_{n+1}\geq \alpha(b_1+b_2+\ldots+b_n)$ for some $\alpha>0$. Then $(a_n)$ is strongly non dli.
\end{proposition}
 \begin{proof}
 Consider the arithmetic sequence $(a_n)$ such that the sequence of ratio $(b_n)$ satisfies the following condition,
 $$
 b_{n+1}\geq \alpha(b_1+b_2+\ldots+b_n) \mbox{ for some } \alpha>0.
 $$
 Now pick any infinite $A\subseteq\N$. Let us write $A=\{s_j+1:j\in\N\}$. Now observe that
 \begin{eqnarray*}
 \bar{d}(L(A)) &=& \limsup_{n\to\infty}\frac{|L(A)\cap[1,n]|}{n} \\
 &\geq& \limsup_{n\to\infty}\frac{|L(A)\cap[1,n_{s_j+1}-1]|}{n_{s_j+1}-1} \\
&\geq& \limsup_{j\to\infty}\frac{b_{s_j+1}}{\sum\limits_{n=1}^{s_j+1}b_{n}}\\
&\geq& \limsup_{j\to\infty} \frac{\alpha}{\alpha+1} = \frac{\alpha}{\alpha+1} > 0.
 \end{eqnarray*}
 Since $A$ was chosen arbitrarily, we can conclude that $(a_n)$ is strongly non dli.
\end{proof}
\begin{example}\label{exsndli1}
Consider the arithmetic sequence $(a_n)$ such that $b_n=2^n$ (i.e., $a_n :=b_1\cdot b_2 \ldots b_n=2^{\frac{n(n+1)}{2}}$). Then we have $b_{n+1}=(b_1+b_2+\ldots+b_n)$. Therefore, in view of Proposition \ref{prosndli}, we conclude that $(a_n)$ is strongly non dli.
\end{example}
Our next result provides a more general view of \cite[Lemma 1.17]{DG2} and will play a key role in establishing Lemma \ref{le2} and Lemma \ref{le3}.
\begin{lemma}\label{le1}
Suppose $(a_n)$ is an arithmetic sequence and $A\subseteq\mathbb{N}$ is non $b$-bounded set with $\bar{d}(L(A))>0$. If $d(L(A'))=0$ holds for each $b$-bounded subset $A'$ of $A$ then there exists a $b$-divergent set $B\subseteq A$ such that $d(L(A\setminus B))=0$.
\end{lemma}
\begin{proof}
For each $m\in\mathbb{N}\setminus\{1\}$, we set $A_m=\{n\in A: b_n=m\}$. Since $d(L(A'))=0$ for each $b$-bounded set $A'\subseteq A$ it is clear that $d(L(A_m))=0$ for all $m\in\mathbb{N}\setminus\{1\}$. Now consider the following situations:
\begin{itemize}
\item First, assume that there exists $n_0\in\mathbb{N}$ such that $A_m$ is finite for each $m>n_0$. We set $B=\bigcup_{m=n_0+1}^{\infty} A_m$. Then it is evident that $B$ is $b$-divergent and $d(L(A)\setminus L(B))=d(\bigcup_{m=2}^{n_0} L(A_m))=0$. Consequently, $d(L(A\setminus B))=0$.\\
\item Now, without any loss of generality, let us assume that $A_m$ is infinite for each $m\in\mathbb{N}\setminus\{1\}$. Since $(L(A_m))$ is a sequence of density zero sets, there exists $C\subseteq \mathbb{N}$ such that $d(C)=0$ and $L(A_m)\setminus C$ is finite for each $m\in\mathbb{N}\setminus\{1\}$ (see \cite{S} for the explicit construction of this set). Therefore, for each $m\geq 2$, there exists $t_m\in\N$ such that
$$
L(A_m)\setminus C \subseteq \bigcup\limits_{i=1}^{t_m} [n_{k^{(m)}_i-1},n_{k^{(m)}_i}-1]
$$
and
$$
\big(L(A_m)\setminus C\big) \cap  [n_{k^{(m)}_i-1},n_{k^{(m)}_i}-1]\neq \emptyset \mbox{ for each }  i\in \{1,2,\ldots,t_m\}.
$$
In view of Remark \ref{uconr2} and the fact that $(A_m)$ forms a partition of $\mathbb{N}$, it is easy to see that $(L(A_m))$ also forms a partition of $\mathbb{N}$. We now set
$$A_m'=A_m\setminus\{k_i^{(m)}:1\leq i\leq t_m\} \mbox{  and  } C'=\bigcup_{m=2}^{\infty} L(A_m').
$$
Clearly $C'\subseteq C$. From Remark \ref{uconr2}, it follows that $C'=L(C'')$, where $C''=\bigcup_{m=2}^{\infty}A_m'$. Consequently, observe that $d(L(C''))=0$ and $L(A_m)\setminus L(C'')$ is finite for each $m\geq 2$. This ensures that $A_m\setminus C''$ is finite for each $m\geq 2$. Now, put $B=A\setminus C''$. Since $d(L(C''))=0$, we have $d(L(A\setminus B))=d(L(A)\setminus L(B))=0$. Let $l\in\mathbb{N}\setminus \{1\}$ be arbitrary. We set
$$n_{l}=\max\{n\in\mathbb{N}:n\in A_m\setminus C''~\mbox{and}~m\leq l\},$$
(note that such an element $n_l$ exists since each $A_m\setminus C''$ is finite).
Now for each $n\in B$ with $n>n_{l}$, we then have $b_n>l$. Since $l$ was arbitrarily chosen, we can conclude that $B$ is $b$-divergent.
\end{itemize}
\end{proof}
In order to establish our main result, showing that $t^{s}_{(d_n)}(\mathbb{T})$ must be countable for strongly non-dli sequences (Theorem \ref{sconthmain2})  the main thing is to find conditions for $x\in\T$ which places them outside $t^{s}_{(d_n)}(\mathbb{T})$ and this will be specified in the next two lemmas.
\begin{lemma}\label{le2}
Let $(a_n)$ be an arithmetic sequence, and let $x\in\mathbb{T}$ be such that $supp(x)$ is co-finite. If $\bar{d}(L(supp(x)\setminus\supp_{q}(x)))>0$ then $x\notin t^{s}_{(d_n)}(\mathbb{T})$.
\end{lemma}
\begin{proof}
Let us choose $x\in\mathbb{T}$ for which $supp(x)$ is co-finite and $\bar{d}(L(supp(x)\setminus\supp_{q}(x))>0$. We set $A=supp(x)\setminus\supp_{q}(x)$. Then it is clear that $A$ is infinite since $\supp_{q}(x)$ cannot be co-finite (see Fact \ref{lemmanew}). Therefore, for each $k-1\in A$ and $r\in [1,b_k-1]$, we have
\begin{equation}\label{eq1}
\frac{rc_k}{b_k}\leq r\{a_{k-1}x\}\leq \frac{r(c_k+1)}{b_k}
\end{equation}
Next, for any $m_0,n_0\in\mathbb{N}$ with $m_0>9$ and $  n_0>12$, we define the following sets:
\begin{eqnarray*}
A_1 &=& \{n\in A:0<\frac{c_n}{b_n}<\frac{1}{m_0}\},\\
A_2 &=& \{n\in A: 1-\frac{1}{n_0}<\frac{c_n}{b_n}<1\},\\
 \mbox{  and,  }~
 A_3 &=& \{n\in A:
 \frac{1}{m_0}\leq\frac{c_n}{b_n}\leq 1-\frac{1}{n_0}\}.
 \end{eqnarray*}
Note that $(A_1, A_2, A_3)$ is a partition of $A$. Since $\bar{d}(L(A))>0$, we can derive that  $\sum\limits_{i=1}^{3}\bar{d}( L(A_i))>0$ which implies that $\bar{d}( L(A_i))>0$ for some $i\in\{1,2,3\}$. So, let us now consider the following cases:
\begin{itemize}
    \item [\textbf{Case I}:] Assume that $\bar{d}(L(A_1))>0$.
    %First, note that for each $k\in A_1$, we have $$\left[\frac{b_k}{m_0c_k},\frac{2b_k}{m_0c_k}\right]\cap\mathbb{N}\neq\emptyset~\mbox{ (since }~\frac{c_k}{b_k}<\frac{1}{m_0}).$$
  For each $k\in A_1$, let us define
  $$
  B_k=\bigcup_{m=0}^{\lfloor \frac{c_k}{m_0}\rfloor}\left[n_{k-1}+\lfloor(m+\frac{1}{m_0})\frac{b_k}{c_k}\rfloor, n_{k-1}+\lfloor(m+\frac{4}{m_0})\frac{b_k}{c_k}\rfloor-1\right].
  $$
  Observe that for each $k\in A_1$, the set $B_k$ is the union of $\lfloor \frac{c_k}{m_0}\rfloor+1$ number of mutually disjoint non-empty sub-intervals. Note that the length of each such sub-interval
  \begin{align*}
      &=\lfloor(m+\frac{4}{m_0})\frac{b_k}{c_k}\rfloor-1-\lfloor(m+\frac{1}{m_0})\frac{b_k}{c_k}\rfloor\\
      &\geq \frac{3b_k}{m_0c_k}-2>\frac{b_k}{m_0c_k}~\mbox{ (since }~\frac{c_k}{b_k}<\frac{1}{m_0}~\mbox{ for each }~k\in A_1).
  \end{align*}\\
  Thus, the total length of $B_k$ must be
  $$
  \geq\frac{b_k}{m_0c_k}(\lfloor\frac{c_k}{m_0}\rfloor+1)\geq \frac{b_k}{m_0^2}.
  $$
  Further observe that, for each $k\in A_1$, we have $B_k\subseteq[n_{k-1},n_{k-1}+b_{k}-2]$. Let $B'=\bigcup_{k\in A_1}B_k$. Observe that $B'\subseteq L(A_1)$, and
\begin{eqnarray*}
\bar{d}(B')&=&\limsup_{n\to\infty}\frac{|B'\cap[1,n]|}{n}\\
  &\geq&\limsup_{\begin{smallmatrix} k \to \infty & \\k\in A_1 \end{smallmatrix}}\frac{|B'\cap[1,n_{k-1}+b_k-2]|}{n_{k-1}+b_k-2}\\
  &\geq& \frac{1}{m_0^2}\limsup_{\begin{smallmatrix} k \to \infty & \\k\in A_1 \end{smallmatrix}}\frac{\displaystyle\sum_{j=1}^{k}(b_j-2)}{n_{k-1}+b_k-2}\\
   &=& \frac{1}{m_0^2}\limsup_{\begin{smallmatrix} k \to \infty & \\k\in A_1 \end{smallmatrix}}\frac{|L(A_1)\cap[1,n_{k-1}+b_k-2]|}{n_{k-1}+b_k-2}\\
  &=&\frac{1}{m_0^2}\bar{d}(L(A_1))>0.
  \end{eqnarray*}
 Let $i\in B'$ be arbitrary. Then $i\in B_k$ for some $k\in A_1$ and so $i=n_{k-1}+r-1$ for some $r\in \left[\lfloor(m+\frac{1}{m_0})\frac{b_k}{c_k}\rfloor+1, \lfloor(m+\frac{4}{m_0})\frac{b_k}{c_k}\rfloor\right]$ and some $m\in [0,\lfloor \frac{c_k}{m_0}\rfloor]$. Then, in view of Eq (\ref{eq1}), we obtain
$$ \frac{rc_k}{b_k}\leq r\{a_{k-1}x\}\leq \frac{r(c_k+1)}{b_k}$$
Since $c_k\geq 1$ and $m\leq \frac{c_k}{m_0}$, we have
 \begin{eqnarray*}
  m+\frac{1}{m_0}\leq r\{a_{k-1}x\}\leq m+\frac{4}{m_0}+\frac{m}{c_k}+\frac{4}{m_0c_k} \leq m+\frac{9}{m_0}\\
     \Rightarrow\frac{1}{m_0}\leq \{r\{a_{k-1}x\}\}\leq \frac{9}{m_0}~\mbox{ (since}~m_0> 9),\\
     \mbox{i.e., }~  \frac{1}{m_0}\leq \{d_{i}x\} \leq \frac{9}{m_0}.
 \end{eqnarray*}
 Since $i\in B'$ was chosen arbitrarily and $\bar{d}(B')>0$, we can conclude that $x\notin t^{s}_{(d_n)}(\mathbb{T})$.\\
 \item[\textbf{Case II}:] Suppose that $\bar{d}(L(A_2))>0$. Note that, for each $k\in A_2$, we have
 $$
 0<\frac{b_k-c_k}{c_k}<\frac{1}{n_{0}}.
 $$
Here, for each $k\in A_2$, let us set
$$
B_k=\bigcup_{m=0}^{\lfloor \frac{(b_k-c_k)}{2n_0}\rfloor}\left[n_{k-1}+\lfloor(m+\frac{8}{n_0})\frac{b_k}{(b_k-c_k)}\rfloor, n_{k-1}+\lfloor(m+\frac{12}{n_0})\frac{b_k}{(b_k-c_k)}\rfloor-1\right].
$$
Observe that $B_k$ is the union of ${\lfloor \frac{(b_k-c_k)}{2n_0}\rfloor}+1$ number of mutually disjoint sub-intervals, each of whose length is equal to
\begin{align*}
  & \lfloor(m+\frac{12}{n_0})\frac{b_k}{(b_k-c_k)}\rfloor-1-\lfloor(m+\frac{8}{n_0})\frac{b_k}{(b_k-c_k)}\rfloor\\
  &\geq (m+\frac{12}{n_0})\frac{b_k}{(b_k-c_k)}-1-1-(m+\frac{8}{n_0})\frac{b_k}{(b_k-c_k)}\\
  &=\frac{4b_k}{n_0(b_k-c_k)}-2>\frac{2b_k}{n_0(b_k-c_k)},
\end{align*}
since $0<\frac{b_k-c_k}{c_k}<\frac{1}{n_{0}}$ for each $k\in A_2$. Hence, the total length of $B_k$ must be
$$
\geq\frac{2b_k}{n_0(b_k-c_k)} \left({\left\lfloor \frac{(b_k-c_k)}{2n_0}\right\rfloor}+1\right)\geq \frac{b_k}{n_0^2}.
$$
Also, for each $k\in A_2$, note that
$$
B_k\subseteq [n_{k-1},n_{k-1}+b_k-2].
$$
Now let us define $B'=\bigcup_{k\in A_2} B_k$.
Then, from the construction of $B'$, it easily follows that
$$
 B'\subseteq L(A_2)~\mbox{ and }~\bar{d}(B')\geq \frac{1}{n_0^2}\bar{d}(L(A_2))>0.
$$
 Pick an $i\in B'$. Then there exists $k\in A_2$ such that $i\in B_k$, i.e., $i=n_{k-1}+r-1$ for some $r\in \left[\lfloor(m+\frac{8}{n_0})\frac{b_k}{(b_k-c_k)}\rfloor+1, \lfloor(m+\frac{12}{n_0})\frac{b_k}{(b_k-c_k)}\rfloor\right]$ and some $m\in [0,\lfloor \frac{(b_k-c_k)}{2n_0}\rfloor]$. Therefore, by Eq (\ref{eq1}), we can write
 \begin{align*}
~&\frac{c_k}{b_k}\leq \{a_{k-1}x\}\leq \frac{c_k+1}{b_k}\\
\Rightarrow~& \frac{(b_k-c_k-1)}{b_k}\leq \|a_{k-1}x\|\leq \frac{(b_k-c_k)}{b_k}\\
\Rightarrow~& \frac{r(b_k-c_k)}{b_k}-\frac{r}{b_k}\leq r\|a_{k-1}x\|\leq \frac{r(b_k-c_k)}{b_k}\\
\Rightarrow~& m+\frac{8}{n_0}-(m+\frac{12}{n_0})\frac{1}{(b_k-c_k)}\leq r\|a_{k-1}x\|\leq m+\frac{12}{n_0}\\
\Rightarrow~& m+\frac{8}{n_0}-\frac{m}{(b_k-c_k)}-\frac{12}{n_0(b_k-c_k)}\leq r\|a_{k-1}x\|\leq m+\frac{12}{n_0}\\
\Rightarrow~& m+\frac{8}{n_0}-\frac{1}{2n_0}-\frac{6}{n_0}\leq r\|a_{k-1}x\|\leq m+\frac{12}{n_0}\\
~&\mbox{ (since }~b_k-c_k\geq 2~\mbox{ and }~\frac{m}{b_k-c_k}\leq\frac{1}{2n_0})\\
\Rightarrow~& \frac{3}{2n_0}\leq \{r\|a_{k-1}x\|\}\leq \frac{12}{n_0}~\mbox{ (since $n_0>12$)}\\
\Rightarrow~& \|d_ix\|=\|r\{a_{k-1}x\}\|=\|r\|a_{k-1}x\|\|\geq\min\left\{\frac{3}{2n_0},1-\frac{12}{n_0}\right\}.
 \end{align*}
 Since $i\in B'$ was picked arbitrarily and $\bar{d}(B')>0$, we deduce that $x\notin t^s_{(d_n)}(\mathbb{T})$.\\

 \item [\textbf{Case III}:] Finally assume that $\bar{d}(L(A_3))>0$. The following two subcases can arise:
 \begin{itemize}
     \item [Subcase (a1):] Assume that there exists a $b$-bounded set $A'\subseteq A_3$ such that $\bar{d}(L(A'))>0$. Then there exists $M>0$ such that $2\leq b_k\leq M$ for each $k\in A'$. Observe that, for each $n\in A'$, we can write
     $$
       \frac{1}{b_k}\leq \frac{c_k}{b_k}\leq \{a_{k-1}x\} \leq \frac{c_k+1}{b_k}\leq \frac{b_k-1}{b_k}
        =1-\frac{1}{b_k}\\
        $$
        $$
         \Rightarrow \frac{1}{M}\leq \{a_{k-1}x\}\leq 1-\frac{1}{M}.
      $$
For $\epsilon\in (0,\frac{1}{2M})$ we consider all such $r\in [1,b_{k}-1]$ with $k\in A'$ such that
$$
m-\epsilon<r\{a_{k-1}x\}<m+\epsilon~\mbox{ for some }~m\in [0,r].
$$
Subsequently, we obtain that
 \begin{align}\label{eq10}
     &\frac{1}{M}+m-\epsilon<(r+1)\{a_{k-1}x\}<m+\epsilon+1-\frac{1}{M}\nonumber\\
   \Rightarrow&~  \frac{1}{2M}+m\leq (r+1)\{a_{k-1}x\}\leq m+1-\frac{1}{2M}\nonumber,\\
   \mbox{i.e., }&~  \frac{1}{2M}\leq \{(r+1)\{a_{k-1}x\}\} \leq 1-\frac{1}{2M}.
 \end{align}
 This means that for any $k\in A'$, there are at least $\lfloor \frac{(b_k-1)}{2}\rfloor$ many $r\in [1,b_{k}-1]$ that satisfy Eq (\ref{eq10}), and we will denote $B_k'$ to represent the set of all such $r\in [1,b_{k}-1]$. For each $k\in A'$, we set  $B_k= \{n_{k-1}+r-1:r\in B_k'\}$. Let us define $B'=\bigcup_{k\in A'}B_k$. It is clear that
 $$
 B'\subseteq L(A')~\mbox{ and }~ \bar{d}(B')\geq \frac{1}{2} \bar{d}(L(A'))>0.
 $$
 Let $i\in B'$ be arbitrary. Then there exists $k\in A'$ such that $i=n_{k-1}+r-1$ for some $r\in B_k'$. Therefore, we have
 $$
 \{d_ix\}=\{r\{a_{k-1}x\}\}\in \left[\frac{1}{2M},1-\frac{1}{2M}\right].
 $$
 Since $\bar{d}(B')>0$ and $i\in B'$ was arbitrary, we can conclude that $x\notin t^{s}_{(d_n)}{(\mathbb{T})}$.
 \item[Subcase (a2):] Let us now assume that there does not exist any $A'\subseteq A_3$ such that $A'$ is $b$-bounded and $\bar{d}(L(A'))>0$. Then, in view of Lemma \ref{le1}, there exists $B\subseteq A_3$ with $d(L(A_3\setminus B))=0$ such that $B$ is $b$-divergent. So, without any loss of generality, we can assume that $b_n\geq 2n_0$ for each $n\in B$.
 Note that, for each $k\in B$, we have
 \begin{align*}
   \frac{1}{m_0} \leq \frac{c_k}{b_k}&\leq \{a_{k-1}x\}\leq \frac{c_k+1}{b_k}\leq 1-\frac{1}{n_0}+ \frac{1}{2n_0}=1-\frac{1}{2n_0}.
 \end{align*}
 For $0<\epsilon<\min\{\frac{1}{2m_0},\frac{1}{4n_0}\}$, we pick all such $r\in [1,b_{k}-1]$ with $k\in B$ such that
$$
m-\epsilon<r\{a_{k-1}x\}<m+\epsilon~\mbox{ for some }~m\in [0,r].
$$
Therefore, observe that
\begin{align}\label{eq11}
     &\frac{1}{m_0}+m-\epsilon<(r+1)\{a_{k-1}x\}<m+\epsilon+1-\frac{1}{2n_0}\nonumber\\
   \Rightarrow&~  \frac{1}{2m_0}+m \leq (r+1)\{a_{k-1}x\} \leq m+1-\frac{1}{4n_0}\nonumber,\\
   \mbox{i.e., }&~  \frac{1}{2m_0}\leq \{(r+1)\{a_{k-1}x\}\}\leq 1-\frac{1}{4n_0} ~\mbox{ (since }~n_0> 12).
 \end{align}
 This entails that for each $k\in B$, there are at least $\lfloor \frac{(b_k-1)}{2}\rfloor$ many $r\in [1,b_{k}-1]$ that satisfy Eq (\ref{eq11}). Let $B_k'$ be the set of all such $r\in [1,b_{k}-1]$. Now, for each $k\in B$, we define $B_k= \{n_{k-1}+r-1:r\in B_k'\}$. Let us set $B'=\bigcup_{k\in B}B_k$. Consequently it is evident that
 $$
 B'\subseteq L(B)~\mbox{ and }~ \bar{d}(B')\geq \frac{1}{2} \bar{d}(L(B))>0.
 $$
 Pick an $i\in B'$. Then there exists $k\in B$ such that $i=n_{k-1}+r-1$ for some $r\in B_k'$. Thus, we get that
 $$
 \{d_ix\}=\{r\{a_{k-1}x\}\}\in \left[\frac{1}{2m_0},1-\frac{1}{4n_0}\right].
 $$
 Since $\bar{d}(B')>0$ and $i\in B'$ was chosen arbitrarily, we can infer that $x\notin t^{s}_{(d_n)}{(\mathbb{T})}$.
 \end{itemize}
\end{itemize}
Finally, in view of Case I, Case II and Case III, we conclude that $x\notin t^{s}_{(d_n)}{(\mathbb{T})}$.
\end{proof}
\begin{lemma}\label{le3}
Let $(a_n)$ be an arithmetic sequence, and let $x\in\mathbb{T}$ be such that $supp(x)$ is infinite but non co-finite. If $\bar{d}(L(supp(x)\setminus(supp(x)-1)))>0$ then $x\notin t^{s}_{(d_n)}(\mathbb{T})$.
\end{lemma}
\begin{proof}
  Let us choose $x\in\mathbb{T}$ such that both $supp(x)$ and $\mathbb{N}\setminus supp(x)$ are infinite and  $\bar{d}(L(supp(x)\setminus(supp(x)-1)))>0$. We set $A=supp(x)\setminus(supp(x)-1)$. Since $\mathbb{N}\setminus\supp(x)$ is infinite, \cite[Claim 3.1]{DI1} ensures that $A$ is infinite. Note that for each $n\in A$, we have $c_n\geq 1$ but $c_{n+1}=0$. Therefore, observe that
  \begin{align}\label{eq12}
      \{a_{n-1}x\}=a_{n-1}\sum_{i=n}^{\infty}\frac{c_i}{a_i}&=\frac{a_{n-1}c_n}{a_n}+{a_{n-1}}\sum_{i=n+2}^{\infty}\frac{c_i}{a_i}\nonumber\\
               &\leq \frac{a_{n-1}c_n}{a_n}+\frac{a_{n-1}}{a_{n+1}}\nonumber\\
               &=\frac{c_n}{b_n}+\frac{1}{b_nb_{n+1}}\leq \frac{c_n+\frac{1}{2}}{b_n},
\end{align}
and,
\begin{equation}\label{eq13}
    \{a_{n-1}x\}=a_{n-1}\sum_{i=n}^{\infty}\frac{c_i}{a_i}\geq \frac{a_{n-1}c_n}{a_n}=\frac{c_n}{b_n}.
\end{equation}
Subsequently, for each $n\in A$, it follows that
\begin{equation}\label{eq14}
\frac{c_n}{b_n}\leq\{a_{n-1}x\}\leq \frac{c_n+\frac{1}{2}}{b_n}< \frac{c_n+1}{b_n}.
\end{equation}
For any fixed $m_0,n_0\in\mathbb{N}$ with $m_0>9$ and $n_0>12$,
we next consider the following sets:
\begin{eqnarray*}
A_1 &=& \{n\in A:0<\frac{c_n}{b_n}<\frac{1}{m_0}\}, \\
A_2 &=& \{n\in A:1-\frac{1}{n_0}<\frac{c_n}{b_n}<1\}, \\
\mbox{  and,  }~
A_3 &=& \{n\in A:\frac{1}{m_0}\leq\frac{c_n}{b_n}\leq 1-\frac{1}{n_0}\}.
\end{eqnarray*}
Since $A=A_1\cup A_2\cup A_3$, we must have $\sum\limits_{i=1}^{3}\bar{d}(L(A_i))>0$. So, the following cases can arise:\\
\begin{itemize}
\item [\textbf{Case I}:] Suppose that $\bar{d}(L(A_1))>0$.
In this case,  we set
$$
B'=\bigcup_{k\in A_1}\bigcup_{m=0}^{\lfloor \frac{c_k}{m_0}\rfloor}\left[n_{k-1}+\lfloor(m+\frac{1}{m_0})\frac{b_k}{c_k}\rfloor, n_{k-1}+\lfloor(m+\frac{4}{m_0})\frac{b_k}{c_k}\rfloor-1\right].
$$
Then, proceeding exactly as in  Case I of Lemma \ref{le2}, we obtain that $\bar{d}(B')>0$ and for each $i\in B'$, we have
  $$
  \frac{1}{m_0}\leq \{d_{i}x\} \leq \frac{9}{m_0}.
  $$
  This ensures that $x\notin t^s_{(d_n)}(\mathbb{T})$.\\
   \item [\textbf{Case II}:] Assume that $\bar{d}(L(A_2))>0$. Here, we set
   $$
B'=\bigcup_{k\in A_2}\bigcup_{m=0}^{\lfloor \frac{(b_k-c_k)}{2n_0}\rfloor}\left[n_{k-1}+\lfloor(m+\frac{8}{n_0})\frac{b_k}{(b_k-c_k)}\rfloor, n_{k-1}+\lfloor(m+\frac{12}{n_0})\frac{b_k}{(b_k-c_k)}\rfloor-1\right].
$$
Then, in view of Case II of Lemma \ref{le2}, we obtain that
$$
B'\subseteq L(A_2)~\mbox{ and }~\bar{d}(B')>0.
$$
 Now, pick any $i\in B'$. Then there exists $k\in A_2$ such that $i\in B_k$, i.e., $i=n_{k-1}+r-1$ for some $r\in \left[\lfloor(m+\frac{8}{n_0})\frac{b_k}{(b_k-c_k)}\rfloor+1, \lfloor(m+\frac{12}{n_0})\frac{b_k}{(b_k-c_k)}\rfloor\right]$ and some $m\in [0,\lfloor \frac{(b_k-c_k)}{2n_0}\rfloor]$. Since $\frac{c_k}{b_k}>\frac{1}{2}$ for each $k\in A_2$, from  Eq (\ref{eq12}) and Eq (\ref{eq13}), it follows that
\begin{align*}
~&\frac{b_k-c_k-\frac{1}{2}}{b_k}\leq \|a_{k-1}x\|\leq \frac{b_k-c_k}{b_k}\\
\Rightarrow~ &\frac{r(b_k-c_k)}{b_k}-\frac{r}{2b_k}\leq r\|a_{k-1}x\|\leq \frac{r(b_k-c_k)}{b_k}\\
\Rightarrow~ &m+\frac{8}{n_0}-(m+\frac{12}{n_0})\frac{1}{2(b_k-c_k)}\leq r\|a_{k-1}x\|\leq m+\frac{12}{n_0}\\
\Rightarrow~ &m+\frac{8}{n_0}-\left(\frac{m}{2(b_k-c_k)}+\frac{6}{n_0(b_k-c_k)}\right)\leq r\|a_{k-1}x\|\leq m+\frac{12}{n_0}\\
\Rightarrow~ &m+\frac{8}{n_0}-\frac{1}{4n_0}-\frac{6}{n_0}\leq r\|a_{k-1}x\|\leq m+\frac{12}{n_0}~\\
&\mbox{ (since }~c_n\leq b_n-1~\mbox{ and }~m\leq \frac{(b_k-c_k)}{2n_0})\\
\Rightarrow~ &\frac{7}{4n_0}\leq \{r\|a_{k-1}x\|\}\leq \frac{12}{n_0}\\
\Rightarrow ~& \|d_ix\|=\|r\{a_{k-1}x\}\|=\|r\|a_{k-1}x\|\|\geq\min\left\{\frac{7}{4n_0},1-\frac{12}{n_0}\right\}.
\end{align*}
Since $i\in B'$ was arbitrary and $\bar{d}(B')>0$, we can conclude that $x\notin t^s_{(d_n)}(\mathbb{T})$.\\
\item[\textbf{Case III}:] Assume that $\bar{d}(L(A_3))>0$.
\begin{itemize}
    \item [$\bullet$] First assume that there exists a $b$-bounded subset $A'\subseteq A_3$ such that $\bar{d}(L(A'))>0$. Then we can find $M>0$ such that $2\leq b_k\leq M$ for all $k\in A'$. Note that, in view of Eq (\ref{eq12}) and Eq (\ref{eq13}), for each $k\in A'$, we must have
    \begin{align*}
        \frac{1}{b_k}&\leq \frac{c_k}{b_k}\leq \{a_{k-1}x\}\leq \frac{c_k+\frac{1}{2}}{b_k}
        \leq\frac{b_k-1+\frac{1}{2}}{b_k}
        = 1-\frac{1}{2b_k} \\
        &\Rightarrow\frac{1}{M}\leq \{a_{k-1}x\}\leq 1-\frac{1}{2M}.
    \end{align*}
    For $\epsilon\in (0,\frac{1}{4M})$ we collect all such $r\in [1,b_{k}-1]$ with $k\in A'$ such that
$$
m-\epsilon<r\{a_{k-1}x\}<m+\epsilon~\mbox{ for some }~m\in [0,r].
$$
For this choice we have
 \begin{align}\label{eq15}
     &~\frac{1}{M}+m-\epsilon<(r+1)\{a_{k-1}x\}<m+\epsilon+1-\frac{1}{2M}\nonumber\\
   \Rightarrow&~  \frac{3}{4M}+m<(r+1)\{a_{k-1}x\}<m+1-\frac{1}{4M}\nonumber,\\
   \mbox{i.e., }&~  \frac{3}{4M}<\{(r+1)\{a_{k-1}x\}\}<1-\frac{1}{4M}.
 \end{align}
 This means that for any $k\in A'$, there are at least $\lfloor \frac{(b_k-1)}{2}\rfloor$ many $r\in [1,b_{k}-1]$ that satisfy Eq (\ref{eq15}), and let us denote by $B_k'$ to represent the set of all such $r\in [1,b_{k}-1]$. For each $k\in A'$, we set  $B_k= \{n_{k-1}+r-1:r\in B_k'\}$. Let us define $B'=\bigcup_{k\in A'}B_k$. Then it is clear that
 $$
 B'\subseteq L(A')~\mbox{ and }~ \bar{d}(B')\geq \frac{1}{2} \bar{d}(L(A'))>0.
 $$
 Pick an $i\in B'$. Choose $k\in A'$ such that $i=n_{k-1}+r-1$ for some $r\in B_k'$. Therefore, we have
 $$
 \{d_ix\}=\{r\{a_{k-1}x\}\}\in \left[\frac{3}{4M},1-\frac{1}{4M}\right].
 $$
 Since $\bar{d}(B')>0$ and $i\in B'$ was arbitrary, we can deduce that $x\notin t^{s}_{(d_n)}{(\mathbb{T})}$.\\

 \item [$\bullet$] Let us now assume that there does not exist any $A'\subseteq A_3$ such that $A'$ is $b$-bounded and $\bar{d}(L(A'))>0$. Therefore, by Lemma \ref{le1}, there exists $B\subseteq A_3$  such that $B$ is $b$-divergent and  $d(L(A_3\setminus B))=0$. Now, utilizing Eq (\ref{eq14}) and following the line of argument as in subcase (a2) of Case III in Lemma \ref{le2}, we can find a subset $B'\subseteq \mathbb{N}$ with $\bar{d}(B')>0$ such that for each $i\in B'$, we will have
 $$
 \{d_ix\}\in \left[\frac{1}{2m_0},1-\frac{1}{4n_0}\right].
 $$
 This also ensures that $x\notin t^{s}_{(d_n)}{(\mathbb{T})}$.
\end{itemize}
\end{itemize}
Finally, in view of Case I, Case II, and Case III, we can conclude that $x\notin t^{s}_{(d_n)}{(\mathbb{T})}$.
\end{proof}
Now we are ready to present the most significant result, namely Theorem \ref{sconthmain2}, which indeed offers a positive solution to Problem \ref{prob2}.

\begin{theorem}\label{sconthmain2}
If $(a_n)$ is strongly non dli then $t^{s}_{(d_n)}{(\mathbb{T})}=t_{(d_n)}{(\mathbb{T})}$.
\end{theorem}
\begin{proof}
Note that for any arithmetic sequence $(a_n)$, we always have $t_{(d_n)}{(\mathbb{T})}\subseteq t^s_{(d_n)}{(\mathbb{T})}$.
We assume that $(a_n)$ is strongly non dli. Consider $x\in\T$ such that $supp(x)$ is infinite.

First assume that $supp(x)$ is cofinite. Then it is evident that $supp(x)\setminus supp_{q}(x)$ is also infinite. Since $(a_n)$ is strongly non dli, we must have $\bar{d}(L(supp(x)\setminus\supp_{q}(x)))>0$. Therefore, Lemma \ref{le2} ensures that $x\not\in t^{s}_{(d_n)}(\mathbb{T})$.

Next, let us consider the case when $supp(x)$ is not cofinite. Then $supp(x)\setminus (supp(x)-1)$ must be infinite. As $(a_n)$ is strongly non dli, we can conclude that $\bar{d}(L(supp(x)\setminus (supp(x)-1)))>0$. Consequently, Lemma \ref{le3} implies that $x\not\in t^{s}_{(d_n)}(\mathbb{T})$.

So, in both cases, we obtain $x\not\in t^{s}_{(d_n)}(\mathbb{T})$ whenever $supp(x)$ is infinite. Thus, \cite[Theorem 2.3]{DG8} ensures that $t^{s}_{(d_n)}{(\mathbb{T})}=t_{(d_n)}{(\mathbb{T})}$.
\end{proof}

As a corollary of this result we get another interesting observation namely Corollary \ref{sconcoromain2} which provides negative solutions to Problem \ref{prob1} and Problem \ref{prob4} while providing a positive solution to Problem \ref{prob3}.
\begin{corollary}\label{sconcoromain2}
If $(a_n)$ is strongly non dli then $t^{s}_{(d_n)}{(\mathbb{T})}$ is countable.
\end{corollary}
\begin{proof}
Follows directly from Theorem \ref{sconthmain2} and \cite[Corollary 2.4, (iii)]{DG8}.
\end{proof}
\begin{example}
Consider the arithmetic sequence $(2^{\frac{n(n+1)}{2}})$. In Example \ref{exsndli1}, we have already seen that $(2^{\frac{n^2+n}{2}})$ is strongly non dli. Therefore, Corollary \ref{sconcoromain2} ensures that the corresponding subgroup $t^s_{(d_n)}{(\T)}$ is countable.
\end{example}

\vspace{.3cm}
\subsection{Characterizable subgroups. \vspace{.3cm} \\}
In 2022, Das and Ghosh have shown that for any arithmetic sequence $(a_n)$, the corresponding statistically characterized subgroup $t^s_{(a_n)}(\T)$ cannot be characterized by any sequence of integers (see \cite[Theorem 2.7]{DG}). This result actually provides the novelty of this notion of statistically characterized subgroups. Also, we have observed that for any $q$-bounded arithmetic sequence $(a_n)$ the subgroup $t^s_{(d_n)}(\T)$ actually coincides with the subgroup $t^s_{(a_n)}(\T)$. As a consequence, we have found examples of statistically characterized subgroups corresponding to the sequence $(d_n)$ that cannot be characterized.
\begin{definition}
    A subgroup $H$ of $\T$ is called characterizable if there exists a sequence of integers $(u_n)$ such that $H=t_{(u_n)}(\T)$.
\end{definition}
However, in an interesting departure, for this non-arithmetic sequence $(d_n)$ we have also seen that the subgroup $t^s_{(d_n)}(\T)$ can be characterizable as well (see Corollary \ref{sconcoromain2}). This naturally leads to the question: for which sequences $(d_n)$ does the subgroup $t^s_{(d_n)}(\T)$ fail to be characterizable?

Before answering this question, let us recall an interesting observation from \cite{DG}.
\begin{lemma}\label{lg1}\cite[Lemma 2.5]{DG}
Let $(a_n)$ be an arithmetic sequence and $(u_n)$ be an increasing sequence of naturals. If $G=\{\frac{1}{a_n}:n\in\N\}\subseteq t_{(u_n)}(\T)$ then $u_n$ must be of the form $a_{k_n}v_n$ where $k_n\to\infty$ and $b_{k_n+1}$ does not divide $v_n$ for any $n\in\N$.
\end{lemma}
\begin{lemma}\label{arbault lemma}
    Let $(a_n)$ be an arithmetic sequence and $(u_n)$ be an increasing sequence of naturals such that $u_n=a_{k_n}v_n$ where $k_n\to\infty$ and $b_{{k_n}+1}$ does not divide $v_n$ for any $n\in\mathbb{N}$. Let $x\in\mathbb{T}$ be such that
    \begin{itemize}
        \item [(i)]
        $ supp(x)=\{k_{s_i}+1:i\in\mathbb{N}\}~\mbox{ and }~c_i=\left\lfloor\frac{b_i}{m_i}\right\rfloor~\mbox{ for all }~ i\in\supp(x)$,
         where $1<m_i\leq b_i$, and
         \item[(ii)] $a_{k_{s_{(i+1)}}}\geq 8u_{s_i}$ for all $i\in\mathbb{N}$.
    \end{itemize}
     Then $m_i$ can be chosen in such a way that $x\notin t_{(u_n)}(\mathbb{T})$.
\end{lemma}
\begin{proof}
Assume that $u_n=a_{k_n}v_n$ where $k_n\to\infty$ and $b_{{k_n}+1}\nmid v_n$ for any $n\in\mathbb{N}$. Pick an $x\in\mathbb{T}$ with the aforesaid properties. Therefore, observe that
\begin{eqnarray}\label{eq 12a}
 \{u_{s_i}x\} &=& \left\{v_{s_i}a_{k_{s_i}}\sum_{t=k_{s_i}+1}^{\infty}\frac{c_t}{b_t}\right\}\nonumber \\
 &=& \left\{c_{k_{s_i}+1}\frac{v_{s_i}}{b_{k_{s_i}+1}}+u_{s_i}\sum_{t=k_{s_{(i+1)}}+1}^{\infty}\frac{c_t}{a_t}\right\}
 \end{eqnarray}
 Since $b_{k_{s_i}+1}\nmid v_{s_i}$ for any $i\in\mathbb{N}$, we get that
\begin{equation*}
    \left\{\frac{ v_{s_i}}{b_{k_{s_i}+1}}\right\}=\frac{l_i}{b_{k_{s_i}+1}}~\mbox{ for some }~ l_i\in\{1,2,...,b_{k_{s_i}+1}-1\}.
\end{equation*}
First, we set $l_i'=b_{k_{s_i}+1}-l_i$. Let us now set
\begin{equation*}
m_{k_{s_i}+1}=
    \begin{cases}
		2l_i &~\mbox{ if }~l_i\leq\frac{b_{k_{s_i}+1}}{2},\\
		2l_i' &~ \mbox{ if }~l_i>\frac{b_{k_{s_i}+1}}{2}.
	\end{cases}
\end{equation*}
So, we consider the following cases:
\begin{itemize}
    \item [\textbf{Case I:}] Assume that $l_i\leq\frac{b_{k_{s_i}+1}}{2}$. Then Condition (i) yields that there exists an $e_i\in\{1,2,...,2l_i-1\}$ such that
    $$
    c_{k_{s_i}+1}=\frac{b_{k_{s_i}+1}-e_i}{2l_i}.
    $$
    Thus, we have
    \begin{eqnarray*}
    \left\{c_{k_{s_i}+1}\frac{v_{s_i}}{b_{k_{s_i}+1}}\right\}&=& \left\{c_{k_{s_i}+1}
    \left\{\frac{v_{s_i}}{b_{k_{s_i}+1}}\right\}\right\}\\
    &=& \left\{c_{k_{s_i}+1}\frac{l_i}{b_{k_{s_i}+1}}\right\}\\
    &=& \left\{\frac{l_ic_{k_{s_i}+1}}{2l_ic_{k_{s_i}+1}+e_i}\right\}\\
    ~~\mbox{   }~\Rightarrow \frac{1}{4} \leq \left\{c_{k_{s_i}+1}\frac{v_{s_i}}{b_{k_{s_i}+1}}\right\}\leq \frac{1}{2}.
    \end{eqnarray*}
    \item[\textbf{Case II:}] Assume that $l_i>\frac{b_{k_{s_i}+1}}{2}$. Again, in view of Condition (i), there exists an $e_i'\in\{1,2,...,2l_i'-1\}$ such that
     $$
    c_{k_{s_i}+1}=\frac{b_{k_{s_i}+1}-e_i'}{2l_i'}.
    $$
    Subsequently,
    \begin{eqnarray*}
    \left\{c_{k_{s_i}+1}\frac{v_{s_i}}{b_{k_{s_i}+1}}\right\}
    &=& \left\{c_{k_{s_i}+1}\frac{l_i}{b_{k_{s_i}+1}}\right\}\\
    &=& \left\{c_{k_{s_i}+1}\frac{b_{k_{s_i}+1}-l_i'}{b_{k_{s_i}+1}}\right\}\\
    &=& \left\{c_{k_{s_i}+1}-\frac{c_{k_{s_i}+1}l_i'}{2l_i'c_{k_{s_i}+1}+e_i'}\right\}\\
    ~~\mbox{   }~\Rightarrow \frac{1}{2} \leq \left\{c_{k_{s_i}+1}\frac{v_{s_i}}{b_{k_{s_i}+1}}\right\}\leq \frac{3}{4}.
    \end{eqnarray*}
\end{itemize}
Thus, in both cases, we get
\begin{equation}\label{eq 12}
\frac{1}{4} \leq \left\{c_{k_{s_i}+1}\frac{v_{s_i}}{b_{k_{s_i}+1}}\right\}\leq \frac{3}{4}.
\end{equation}
Now, in view of Condition (ii), we also have
\begin{equation}\label{eq 13}
 0\leq u_{s_i}\sum_{t=k_{s_{(i+1)}}+1}^{\infty}\frac{c_t}{a_t} \leq \frac{u_{s_i}}{a_{k_{s_{(i+1)}}}} \leq\frac{1}{8}.
\end{equation}
Finally, in view of Eq (\ref{eq 12a}), Eq (\ref{eq 12}), and Eq (\ref{eq 13}), one obtains
$$
\frac{1}{4}\leq\{u_{s_i}x\}\leq \frac{7}{8}.
$$
Hence we can conclude that $x\notin t_{(u_n)}(\mathbb{T})$.
\end{proof}
\begin{definition}
An arithmetic sequence $(a_n)$ is called weakly dense density lifting invariant (in short, weakly dense dli) if for any infinite subset $K$ of  $\N$ there exists an infinite $A\subseteq K$ such that ${d}(L(A-m))=0$ for each $m\in\N\cup \{0\}$.
\end{definition}
For an arithmetic sequence $(a_n)$, it is easy to observe that
$$
\mbox{$q$-bounded}  \ \Rightarrow \ dli \ \Rightarrow  \mbox{ weakly dense } dli  \ \Rightarrow  \mbox{ weakly } dli.
$$
\begin{theorem}\label{thnovel}
For any weakly dense dli arithmetic sequence $(a_n)$, the subgroup  $t^{s}_{(d_n)}{(\mathbb{T})}$ cannot be characterized.
\end{theorem}
\begin{proof}
Assume, on the contrary, that there exists a weakly dense dli arithmetic sequence $(a_n)$ for which the subgroup $t^{s}_{(d_n)}{(\mathbb{T})}$ can be characterized. Then there exists an increasing sequence $(u_n)$ of naturals such that
$$
t^{s}_{(d_n)}{(\mathbb{T})}\subseteq t_{(u_n)}{(\mathbb{T})}.
$$
We set $G=\{\frac{1}{a_n}:n\in\mathbb{N}\}$. Then observe that $G\subseteq t^{s}_{(d_n)}{(\mathbb{T})}$ since each $x\in G$ has a finite support. Therefore, Lemma \ref{lg1} ensures that $u_n=a_{k_n}v_n$ where $k_n\to\infty$ and $b_{{k_n}+1}\nmid v_n$ for any $n\in\mathbb{N}$. Since $(a_n)$ is weakly dense dli, we choose a subsequence $(u_{s_i})$ of $(u_n)$ with the following properties:
\begin{itemize}
    \item [(i)] $u_{s_1}=u_1$,
    \item [(ii)] $|k_{s_{(i+1)}}-k_{s_i}|\to\infty$, and
    \item[(iii)] $a_{k_{s_{(i+1)}}}\geq 8u_{s_i}$.
    \item[(iv)] $d(\bigcup_{i=0}^{m+1}L(A-i))=0$, where $A=\{k_{s_i}+1:i\in\mathbb{N}\}$.
\end{itemize}
Consider $x\in\mathbb{T}$ such that
$$
supp(x)=A\mbox{ and }c_t=\left\lfloor\frac{b_t}{m_t}\right\rfloor\mbox{ for all } t\in\supp(x)~(\mbox{with}~1<m_t\leq b_t).
$$
Therefore, Lemma \ref{arbault lemma} entails that $x\notin t_{(u_n)}(\mathbb{T})$.

We claim that $x\in t^{s}_{(d_n)}{(\mathbb{T})}$. Let $0<\epsilon<\frac{1}{2}$ be given and choose an $m\in\mathbb{N}$ such that $\frac{1}{2^m}<\epsilon$. We set
$$
B=\bigcup_{i=m+1}^{\infty} [n_{k_{s_i}-m-1},n_{k_{s_i}+1}-1].
$$
Observe that $d(B)=0$ since $B\subseteq\bigcup_{i=0}^{m+1}d(L(A-i))$. Then, for large enough $i\in \mathbb{N}\setminus B$, we  have $i=n_{l}+r-1$ where
$$r\in[1,b_{l+1}-1]~\mbox{ and }~l\notin \displaystyle\bigcup_{j=m+1}^{\infty}[k_{s_j}-m-1,k_{s_j}+1],$$ i.e., $l,l+1,...,l+m+1\notin supp(x)$.
Therefore, we have
\begin{eqnarray*}
    r\{a_lx\}&\leq &ra_l\sum_{i=l+m+2}^{\infty}\frac{c_i}{a_i}\\
              &\leq & \frac{ra_l}{a_{l+m+1}} \leq \frac{r}{b_{l+1}}\frac{a_{l+1}}{a_{l+m+1}}\leq \frac{1}{2^{m}}<\epsilon.
\end{eqnarray*}
Thus, by Eq (\ref{eqr}), we get that $\{ra_lx\}=r\{a_lx\}$. Now, for sufficiently large $i\in \mathbb{N}\setminus B$, we obtain
\begin{align*}
    \{d_ix\}=\{d_{n_l+r-1}x\}=\{ra_lx\}=r\{a_lx\}<\epsilon.
\end{align*}
This entails that $x\in t^{s}_{(d_n)}{(\mathbb{T})} \  -$ which is a contradiction. Therefore, we conclude that the subgroup $t^{s}_{(d_n)}{(\mathbb{T})}$ is not characterizable.
\end{proof}
The following question is an easy consequence of our previous result, namely, Theorem \ref{thnovel}.
\begin{problem}\label{probnew1}
Is $t^s_{(d_n)}(\T)$ characterizable whenever $(a_n)$ is not weakly dense dli?
\end{problem}
We have observed that if $(a_n)$ is weakly dli then $|t^s_{(d_n)}(\T)|=\mathfrak{c}$. Additionally, there exist examples of $(a_n)$ (which are obviously not weakly dli) such that $t^s_{(d_n)}(\T)$ is countable. This leads us to conclude the article with the following open problem:
\begin{problem}\label{probnew2}
Is $t^s_{(d_n)}(\T)$ countable whenever $(a_n)$ is not weakly dli?
\end{problem}

Since countable subgroups of $\T$ are characterizable \cite{DK}, a positive solution to Problem \ref{probnew2} also provides a partial solution to Problem \ref{probnew1}.\\

%%%%%%%%%%%%%%%%%%%%%%%%%%%%%%%%%%%%%%%%%%%%%%%%%%%%%%%%%%%%%%%%%%%%%%%%%%%%
\noindent{\textbf{Acknowledgement:}} The first author as PI and the second author as RA are thankful to SERB(DST) for the CRG project (No. CRG/2022/000264) and the first author is also thankful to SERB for the MATRICS project (No. MTR/2022/000111) during the tenure of which this work has been done. \\

\noindent{\textbf{Authorship contribution statement:}} Pratulananda Das: Writing – review and editing, Supervision. Ayan Ghosh: Writing – review and editing, Conceptualization. Tamim Aziz: Writing – original draft. \\

\noindent{\textbf{Conflict of Interest:}} The authors state that there is no conflict of interest. \\

\noindent{\textbf{Data Availability:}} Not applicable.
%%%%%%%%%%%%%%%%%%%%%%%%%%%%%%%%%%%%%%%%%%%%%%%%%%%%%%%%%%%%%%%%%%%%%%%%%%%%%


\begin{thebibliography}{100}

\bibitem{A1} J. Arbault, Sur l'ensemble de convergence absolue d'une s\' erie trigonom\' etrique., Bull. Soc. Math. Fr. 80 (1952),  253--317.

%\bibitem{Ar1} D. L. Armacost, On pure subgroups of LCA groups, Proc. Amer. Math. soc., 45(3), (1974), 414 - 418.

%\bibitem{A} D. L. Armacost, The structure of locally compact abelian groups, Monographs and Textbooks in Pure and Applied Mathematics, Marcel Dekker Inc. New York 68 (1981).

\bibitem{AT} A. Arhangel’skii , M. Tkachenko, Topological Groups and Related Structures: An Introduction to Topological Algebra, Atlantis Press, Paris, 2008.

\bibitem{AD} L. Außenhofer, D. Dikranjan, Locally quasi-convex compatible topologies on locally compact abelian groups, Mathematische Zeitschrift, 296 (2020), 325-351.

%\bibitem{ADB} L. Außenhofer, D. Dikranjan, A. G. Bruno, Topological Groups and the Pontryagin-van Kampen Duality: An Introduction, Berlin, Boston: De Gruyter, 2022.

\bibitem{BDBW} G. Barbieri, D. Dikranjan, A. Giordano Bruno, H. Weber, Dirichlet sets vs characterized subgroups, Topol. Appl. 231, 50--76 (2017).

\bibitem{BDK}   M. Balcerzak, K. Dems, A. Komisarski, Statistical convergence and ideal convergence for sequences of functions, J. Math. Anal. Appl., 328(1) (2007), 715--729.

%\bibitem{BDMW} G. Babieri, D. Dikranjan,  C, Milan, H. Weber, Answer to Raczkowski's quests on converging sequences of integers, Topol. Appl. 132(1) (2003), 89--101.

\bibitem{BDMW1} G. Babieri, D. Dikranjan,  C, Milan, H. Weber, Topological torsion related to some recursive sequences of integers, Math. Nachr. 281(7) (2008), 930--950.

%\bibitem{Biro} A. B\' \i r\' o, Characterizations of groups generated by Kronecker sets, J. de Th. des Nom. de Bordeaux 19(3) (2007), 567--582.

%\bibitem{CDK} A. Caserta, G. Di Maio, Lj. D. R. Ko\v cinac, Statistical convergence in function spaces, Abstr. Appl. Anal.(2011) 420419.

%\bibitem{DK} G. Di Maio, L.D.R. Ko\v cinac, Statistical convergence in topology, Topol. Appl. 156(1) (2008) 28Ã45.

\bibitem{BDS} A. B\' \i r\' o, J.M. Deshouillers, V.T.  S\' os, Good approximation and characterization of subgroups of $\R/\Z$, Studia Sci. Math. Hungar. 38 (2001), 97--113.

%\bibitem{B1} J.-P. Borel,  Sous-groupes de $\R$ li\' es \' a  la r\' epartition modulo 1 de suites, Ann. Fac. Sci. Toulouse Math. 5(3--4) (1983), 217--235.

%\bibitem{B2} J.-P. Borel,  Sur certains sous-groupes de R li\' es \' a la suite des factorielles, Colloq. Math. 62(1) (1991), 21--30.

%\bibitem{B} J. Braconnier, Sur les groupes topologiques primaries, C.R. Acad. Sci. Paris 218 (1944), 304--305.

\bibitem{Bu1} R. C. Buck, The measure theoretic approach to density, Amer. J. Math. 68 (1946), 560--580.

%\bibitem{Bu2} R.C. Buck, Generalized asymptotic density, Amer. J. Math. 75 (1953), 335--346.

%\bibitem{Bu} L. Bukovsk\' y, The Structure of the Real Line, Birkh\" auser Basel (2011).

\bibitem{BuKR} L. Bukovsk\' y,  N. Kholshchevnikova, N.N., M. Repick\'  y, Thin sets in harmonic analysis and infinite combinatorics, Real Anal. Exch. 20 (1994/1995), 454--509.

\bibitem{CKG} J. Connor, V. Kadets, M. Ganichev,  A Characterization of Banach Spaces with Separable Duals via Weak Statistical Convergence, J. Math. Anal. Appl.,  244 (2000), 251--261.

\bibitem{DG2} P. Das, A. Ghosh, Solution of a general version of Armacost's problem on topologically torsion elements, Acta Math. Hungar., 164(1) (2021), 243--264.

\bibitem{DG} P. Das, A. Ghosh, On a new class of trigonometric thin sets extending Arbault sets, Bul. Sci. Math., 179 (2022), 103157.

\bibitem{DG8} P. Das, A. Ghosh, Characterized subgroups related to some non-arithmetic sequence of integers, Mediterr. J. Math., 21 (2024), 164.

\bibitem{DG9} P. Das, A. Ghosh, Statistically characterized subgroups related to some non-arithmetic sequence of integers,  Expo. Math., 43 (2025), 125653.

\bibitem{D_NYC} D. Dikranjan, Topologically torsion elements of topological groups, Topol. Proc. 26 (2001--2002), 505--532.

%\bibitem{D_Cl} D. Dikranjan, Closure operators related to von Neumannâ??s kernel, Topol. Appl. 153(11) (2006), 1930--1955.

\bibitem{DDB} D. Dikranjan, P. Das, K. Bose, Statistically characterized subgroups of the circle, Fund. Math., 249 (2020), 185-209.

%\bibitem{DG} D. Dikranjan, S.S. Gabriyelyan, On characterized subgroups of compact abelian groups, Topol. Appl. 160 (2013), 2427--2442.

%\bibitem{DGI} D. Dikranjan, A. Giordano Bruno and D. Impieri, Characterized subgroups of topological abelian groups, Axioms 4 (2015), 459--491.

\bibitem{DI1} D. Dikranjan and D. Impieri, Topologically torsion elements of the circle group, Commun. Algebra, 42 (2014),  600--614.

%\bibitem{DI2}  D. Dikranjan and D. Impieri, Question on the Borel complexity of characterized subgroup of the %compact
%abelian groups. Quest. Answ. Gen. Topol. 32,  (2014) 127--144.

%\bibitem{DI3} D. Dikranjan and D. Impieri,  On the Borel complexity of characterized subgroup of the compact abelian groups, Topol. Appl. 201 (2016), 372--387.

\bibitem{DK}   D. Dikranjan, K. Kunen, Characterizing countable subgroups of compact abelian groups, J. Pure Appl. Algebra 208 (2007), 285--291.

%\bibitem{DMT} D. Dikranjan, C. Milan and A. Tonolo, A characterization of the MAP abelian groups, J. Pure Appl. Algebra 197 (2005), 23--41.

\bibitem{DPS} D. Dikranjan, Iv. Prodanov and L. Stoyanov, Topological Groups: Characters, Dualities and Minimal Group Topologies, Pure and Applied Mathematics, Marcel Dekker Inc. New York (1989).

\bibitem{DR} D. Dikranjan, R. Di Santo, Answer to Armacost's quest on topologically torsion elements of the circle group, Comm. Algebra, 32 (2004), 133--146 .

\bibitem{DDBH} D. Dikranjan, R. Di Santo, A. Giordano Bruno and H. Weber, An introduction to $\mathcal{I}$-characterized subgroups of the circle, Topology Appl., 2025, to appear.

\bibitem{MK}  G. Di Maio, Lj.D.R. Ko$\check{c}$inac, Statistical convergence in topology, Topology Appl. 156 (2008), 28--45.

%\bibitem{E} H. Eggleston, Sets of fractional dimensions which occur in some problems of number theory, Proc. Lond. Math. Soc. 54(2) (1952), 42--93.

\bibitem{El}  P. Elia\v s, A classification of trigonometrical thin sets and their interrelations, Proc. Amer. Math. Soc. 125(4) (1997), 1111--1121.

\bibitem{F} H. Fast, Sur la convergence statistique, Colloq. Math. 2 (1951), 241--244.

\bibitem{FGT} A. Faisant, G. Grekos, V. Toma, On the statistical variation of sequences, J. Math. Anal. Appl., 306 (2) (2005), 432--439.

\bibitem{Fr} J. A. Fridy, On statistical convergence, Analysis 5(4) (1985), 301--313.

%\bibitem{Ga} S. Gabriyelyan, Characterizable groups: some results and open questions. Topol. Appl. 159 (2012), 2378--2391.

%\bibitem{GIMP} D. Georgiou, S. Iliadis, A. Megaritis and G. A. Prinos, Ideal-convergence classes, Topology Appl. 222 (2017), 217--226.

%\bibitem{HK1} J. Hart, K. Kunen, Limits in function spaces and compact groups, Topol. Appl. 151(1â??3) (2005), 157--168.

%\bibitem{HK2} J. Hart, K. Kunen, Limits in compact abelian groups, Topol. Appl. 153(7) (2006), 991--1002.

\bibitem{I} D. Impieri, Characterized subgroups, PhD Thesis, July 2015.

\bibitem{Ka} S. Kahane, Antistable classes of thin sets in harmonic analysis, Illinois J. Math. 37 (1993), 186-223.

%\bibitem{KL} C. Kraaikamp, P. Liardet, Good approximations and continued fractions, Proc. Amer. Math. Soc., 112 (2) (1991), 303--309.

%\bibitem{KN} L. Kuipers , H. Niederreiter, Uniform distribution of sequences, Pure and Applied Mathematics, Wiley-Interscience, New York (1974).

%\bibitem{L} G. Larcher, A convergence problem connected with continued fractions, Proc. Amer. Math. Soc. 103(3) (1988), 718--722.

%\bibitem{PL} P. Loth, Compact topologically torsion elements of topological abelian groups, Rend. Sem. Mat. Univ. Padova 113 (2005), 117--123.

%\bibitem{MSC} D.S. Mitrinovi\' c, J. S\' andor, B. Crstici, Handbook of Number Theory, Mathematics and Its Applications, Kluwer Academic Publishers Group, Dordrecht 351 (1996).

\bibitem{R} L. Robertson, Connectivity, divisibility, and torsion, Trans. Amer. Math. Soc. 128 (1967) 482--505.

\bibitem{S} T. \v Sal\' at, On statistically convergent sequences of real numbers, Mathematica Slovaca, 30(2) (1980), 139--150.

%\bibitem{DD} R. Di Santo and D. Dikranjan, Answer to Armacost's quest on topologically torsion elements of the circle group, Commun. Algebra 32 (2004), 133--146.

\bibitem{DGDis} R. Di Santo,  D. Dikranjan and A. Giordano Bruno, Characterized subgroups of the circle group, Ric. Mat. 67(2) (2018), 625--655.

%\bibitem{Sr} S.M. Srivastava,  A course on Borel sets. Graduate text in Mathematics, Springer (180) (1998).

\bibitem{St} H. Steinhaus, Sur la convergence ordinaire et la convergence asymptotique, Colloq. Math. 2 (1951), 73--74.

\bibitem{W} H. Weyl, \" Uber die Gleichverteilung von Zahlen mod. Eins., Math. Ann. 77(3) (1916), 313--352.

\bibitem{Wi} R. Winkler, Ergodic group rotations, Hartman sets and Kronecker sequences (dedicated to Edmund Hlawka on the occasion of his $85^{th}$ birthday), Monatsh. Math. 135, No.4, 333-343 (2002).

\bibitem{Z} A. Zygmund, Trigonometric Series, vols. I,II, Cambridge University Press, Cambridge, New York, Melbourne (1977).
\end{thebibliography}
\end{document}